\documentclass[10pt,reqno,a4paper]{amsart}
\usepackage{amssymb,dsfont,enumerate}
\usepackage[pagebackref]{hyperref}
\usepackage[portrait,a4paper,margin=3cm]{geometry}

\title[non-Hermitian heavy tailed random matrices]{Spectrum of non-Hermitian \\heavy tailed random matrices}

\date{Expanded version of a paper published in \href{http://dx.doi.org/10.1007/s00220-011-1331-9}{Communications in Mathematical Physics, 307, 513-560 (2011)}}

\author{Charles Bordenave}
\address[Ch.~Bordenave]{IMT UMR 5219 CNRS and Universit\'e Paul-Sabatier Toulouse III, France}
\email{charles.bordenave(at)math.univ-toulouse.fr}
\urladdr{http://www.math.univ-toulouse.fr/~bordenave/}

\author{Pietro Caputo}
\address[P.~Caputo]{Dipartimento di Matematica,
  Universit\`a Roma Tre, Italy}
\email{caputo(at)mat.uniroma3.it}
\urladdr{http://www.mat.uniroma3.it/users/caputo/}

\author{Djalil Chafa\"\i}
\address[D.~Chafa{\"{\i}}]{LAMA UMR 8050 CNRS and Universit\'e Paris-Est Marne-la-Vall\'ee, France}
\email{djalil(at)chafai.net} 
\urladdr{http://djalil.chafai.net/}



\newtheorem{thm}{Theorem}[section]%
\newtheorem{cor}[thm]{Corollary}%
\newtheorem{prop}[thm]{Proposition}%
\newtheorem{lem}[thm]{Lemma}%

\newtheorem{defi}[thm]{Definition}%
\newcommand{\dC}{\mathbb{C}}\newcommand{\dD}{\mathbb{D}}
\newcommand{\dE}{\mathbb{E}}
\newcommand{\dH}{\mathbb{H}}

\newcommand{\dN}{\mathbb{N}}
\newcommand{\dP}{\mathbb{P}}
\newcommand{\dR}{\mathbb{R}}

\newcommand{\dZ}{\mathbb{Z}}

 \newcommand{\bbN}{{\mathbb N}} 
  
 \newcommand{\bbR}{{\mathds R}}

\newcommand{\cD}{\mathcal{D}}

\newcommand{\cI}{\mathcal{I}}\newcommand{\cJ}{\mathcal{J}}

\newcommand{\cM}{\mathcal{M}}

\newcommand{\cX}{\mathcal{X}}

\newcommand{\ABS}[1]{{{\left| #1 \right|}}} 
\newcommand{\BRA}[1]{{{\left\{#1\right\}}}} 
\newcommand{\NRM}[1]{{{\left\| #1\right\|}}} 
\newcommand{\PAR}[1]{{{\left(#1\right)}}} 
\newcommand{\SBRA}[1]{{{\left[#1\right]}}} 
\newcommand{\LIP}[1]{{\|#1\|_{\mathrm{Lip}}}} 
\renewcommand{\a}{\alpha} 
\renewcommand{\b}{\beta}  
 
\renewcommand{\d}{\delta}  
\newcommand{\e}{\varepsilon}
 
\newcommand{\g}{\gamma}

\renewcommand{\o}{\varnothing}

\renewcommand{\r}{\rho}

\newcommand{\G}{\Gamma}

\newcommand{\si}{\sigma}

\newcommand{\thsp}{\thinspace}
\newcommand{\tc}{\thsp | \thsp}

\renewcommand{\ij}{{i,j}}
\renewcommand{\leq}{\leqslant}             
\renewcommand{\geq}{\geqslant}             
 
\newcommand{\wt}{\widetilde}
\newcommand{\ind}{\mathds{1}}

\newcommand{\tr}{\mathrm{tr}}
\newcommand{\pwit}{\mathrm{PWIT}}

\newcommand{\supp}{\mathrm{supp}}

\newcommand{\dist}{\mathrm{dist}}
\newcommand{\vect}{\mathrm{span}} 
\newcommand{\cnu}{\check{\nu}} 
\newcommand{\nep}[1]{e^{#1}}

\renewcommand{\Im}{\mathfrak{Im}}

\newcommand{\weak}{\rightsquigarrow}

\numberwithin{equation}{section}

\begin{document}

\begin{abstract}
  Let $(X_{jk})_{j,k\geq1}$ be i.i.d.\ complex random variables such that
  $\left|X_{jk}\right|$ is in the domain of attraction of an $\a$-stable law,
  with $0<\alpha<2$. Our main result is a heavy tailed counterpart of Girko's
  circular law. Namely, under some additional smoothness assumptions on the
  law of $X_{jk}$, we prove that there exist a deterministic sequence $a_n\sim
  n^{1/\a}$ and a probability measure $\mu_\alpha$ on $\mathbb{C}$ depending
  only on $\alpha$ such that with probability one, the empirical distribution
  of the eigenvalues of the rescaled matrix $(a_n^{-1}X_{jk})_{1\leq j,k\leq
    n}$ converges weakly to $\mu_\alpha$ as $n\to\infty$. Our approach
  combines Aldous \& Steele's objective method with Girko's Hermitization
  using logarithmic potentials. The underlying limiting object is defined on a
  bipartized version of Aldous' Poisson Weighted Infinite Tree. Recursive
  relations on the tree provide some properties of $\mu_\a$. In contrast with
  the Hermitian case, we find that $\mu_\a$ is not heavy tailed.
\end{abstract}

\maketitle

{\footnotesize\tableofcontents}

\section{Introduction}

The \emph{eigenvalues} of an $n\times n$ complex matrix $M$ are the roots in
$\dC$ of its characteristic polynomial. We label them
$\lambda_1(M),\ldots,\lambda_n(M)$ so that
$|\lambda_1(M)|\geq\cdots\geq|\lambda_n(M)|\geq0$. We also denote by
$s_1(M)\geq\cdots\geq s_n(M)$ the \emph{singular values} of $M$, defined for
every $1\leq k\leq n$ by $s_k(M):=\lambda_k(\sqrt{MM^*})$ where
$M^*=\overline{M}^\top$ is the conjugate transpose of $M$. We define the
empirical spectral measure and the empirical singular values measure as
\[
\mu_M = \frac 1 n \sum_{k=1} ^n \delta_{\lambda_k (M)} %
\quad \text{and } \quad %
\nu_M = \frac 1 n \sum_{k=1} ^n \delta_{s_k (M)}.
\]
Let $(X _{ij})_{ i , j \geq 1}$ be i.i.d.\  complex random variables with
cumulative distribution function $F$. Consider the matrix $X = (X_{ij}) _{ 1
  \leq i, j \leq n}$. Following Dozier and Silverstein
\cite{MR2322123,MR2326242}, if $F$ has finite positive variance $\sigma^2$,
then for every $z\in\mathbb{C}$, there exists a probability measure
$\mathcal{Q}_{\sigma,z}$ on $[0,\infty)$ depending only on $\sigma$ and $z$,
with explicit Cauchy-Stieltjes transform, such that a.s. (almost surely)
\begin{equation}\label{eq:QCLT}
  \nu_{\frac{1}{\sqrt{n}}X-zI}%
  \underset{n\to\infty}{\weak}%
  \mathcal{Q}_{\sigma,z}
\end{equation} 
where $\weak$ denotes the weak convergence of probability measures. The proof
of \eqref{eq:QCLT} is based on a classical approach for Hermitian random
matrices with bounded second moment: truncation, centralization, recursion on
the resolvent, and cubic equation for the limiting Cauchy-Stieltjes transform.
In the special case $z=0$, the statement \eqref{eq:QCLT} reduces to the
quarter-circular law theorem (square version of the Marchenko-Pastur theorem,
see \cite{marchenko-pastur,MR0467894,MR862241}) and the probability measure
$Q_{\sigma,0}$ is the quarter-circular law with Lebesgue density
\begin{equation}\label{eq:QCDensity}
  x\mapsto \frac{1}{\pi\sigma^2}\sqrt{4\sigma^2-x^2}\ind_{[0,2\sigma]}(x).
\end{equation}
Girko's famous circular law theorem \cite{MR1454343} states under the same
assumptions that a.s.\
\begin{equation}\label{eq:CLT}
\mu_{\frac{1}{\sqrt{n}}X}%
  \underset{n\to\infty}{\weak} %
\mathcal{U}_\sigma
\end{equation} 
where $\mathcal{U}_\sigma$ is the uniform law on the disc
$\{z\in\dC;|z|\leq\sigma\}$. This statement was established through a long
sequence of partial results
\cite{MR0220494,MR773436,MR2130247,MR841088,MR1437734,MR1454343,MR1428519,MR2191234,bai-silverstein-book,1687963,gotze-tikhomirov-new,MR2409368,MR2722794},
the general case \eqref{eq:CLT} being finally obtained by Tao and Vu
\cite{MR2722794} by using Girko's Hermitization with logarithmic potentials
and uniform integrability, the convergence \eqref{eq:QCLT}, and polynomial
bounds on the extremal singular values.

\subsection{Main results}

The aim of this paper is to investigate what happens when $F$ does
not have a finite second moment. We shall consider the following hypothesis:

\begin{enumerate}
\item[(H1)] there exists a slowly varying function $L$ (i.e.\
  $\lim_{t\to\infty}L(x\,t)/L(t) = 1$ for any $x>0$) and a real number
  $\alpha\in(0,2)$ such that for every $t\geq1$
  \[
  \dP ( |X_{11}| \geq t ) = \int_{\{z\in\dC;|z| \geq t\}}\!dF(z) =
  L(t)t^{-\alpha},
  \]
  and there exists a probability measure $\theta$ on the unit circle
  $\mathbb{S}^1:=\{z\in\mathbb{C};|z|=1\}$ of the complex plane such that for
  every Borel set $D\subset \mathbb{S}^1$,
  \[
  \lim_{t \to \infty} \dP \left( \frac{X_{11}}{|X_{11}|} \in D \Bigm| |X_{11}
    | \geq t \right) = \theta(D).
  \]
\end{enumerate}
Assumption (H1) states a complex version of the classical
criterion for the domain of attraction of a real $\alpha$-stable law, see
e.g.\ Feller \cite[Theorem IX.8.1a]{Feller}. For instance, if
$X_{11}=V_1+iV_2$ with $i=\sqrt{-1}$ and where $V_1$ and $V_2$ are independent
real random variables both belonging to the domain of attraction of an
$\alpha$-stable law then (H1) holds. When (H1) holds, we define the sequence
\[
a_n := \inf\{a > 0 \text{ s.t. } n \dP(|X_{11}| \geq a) \leq 1\}
\]
and (H1) implies that $\lim_{n\to\infty}n \dP(|X_{11}| \geq a_n ) =
\lim_{n\to\infty} n a_n^{-\alpha} L(a_n)=1$. It follows then classically that
$a_n = n^{1/\alpha}\ell(n)$ for every $n\geq1$, for some slowly varying function
$\ell$. The additional possible assumptions on $F$ to be considered in the
sequel are the following:
\begin{enumerate}
\item[(H2)] $\dP(|X_{11}|\geq t) \sim_{t \to \infty} c\, t^{-\a}$ for some
  $c>0$ (this implies $a_n\sim_{n\to\infty}c^{1/\a}n^{1/\a}$)
\item[(H3)] $X_{11}$ has a bounded probability Lebesgue density on $\dR$ or on
  $\dC$.
\end{enumerate}
One can check that (H1-H2-H3) hold e.g.\ when $|X_{11}|$ and $X_{11}/|X_{11}|$ are
independent with $|X_{11}|=|S|$ where $S$ is real symmetric $\alpha$-stable. 
Another basic example is given by $X_{11}=\varepsilon
W^{-1/\alpha}$ with $\varepsilon$ and $W$ independent such that $\varepsilon$ takes values in 
$S^{1}$ and 
$W$ is uniform on $[0,1]$.

For every $n\geq1$, let us define the i.i.d.\ $n\times n$ complex matrix $A
=A_n$ by
\begin{equation}\label{eq:defAn}
A_{ij} := a_n ^{-1}X_{ij}
\end{equation}
for every $1\leq i,j\leq n$. Our first result concerns the singular values of
$A-zI$, $z\in\dC$.

\begin{thm}[Singular values]\label{th:mpz}
  If (H1) holds then for all $z \in \dC$, there exists a probability measure
  $\nu_{\alpha,z}$ on $[0,\infty)$ depending only on $\alpha$ and $z$ such
  that a.s.\
  \[
  \nu_{A-zI} \underset{n\to\infty}{\weak} \nu_{\alpha,z}.
  \]
\end{thm}

The case $z=0$ was already obtained by Belinschi, Dembo and Guionnet
\cite{BDG}. Theorem \ref{th:mpz} is a heavy tailed version of the
Dozier and Silverstein theorem \eqref{eq:QCLT}. Our main results below give a non-Hermitian version of Wigner's theorem for L\'evy matrices
\cite{BouchaudCizeau,benarous-guionnet,BDG,bordenave-caputo-chafai-ii}, as well as a heavy tailed version of Girko's circular law theorem
\eqref{eq:CLT}. 

\begin{thm}[Eigenvalues]\label{th:girko}
  If (H1-H2-H3) hold then there exists a probability measure $\mu_\alpha$ on
  $\dC$ depending only on $\alpha$ such that a.s.\
  \[
  \mu_{A} \underset{n\to\infty}{\weak} \mu_\alpha.
  \]
\end{thm}


\begin{thm}[Limiting law]\label{th:asydensity}
  The probability distribution $\mu_\alpha$ 
  from theorem \ref{th:girko} is isotropic and has a continuous density. 
  Its density at $z=0$ equals
  \[
  \frac{\Gamma(1+2/\alpha)^2\Gamma(1+\alpha/2)^{2/ \alpha}}%
  {\pi\Gamma(1-\alpha/2)^{2/\alpha}}.
  \]
  Furthermore, up to a multiplicative constant, the density of $\mu_\alpha$ is
  equivalent to
  \[
  |z|^{2 ( \alpha - 1) } e^{- \frac{\alpha}{2} |z|^\alpha}
  \text{ as $\ABS{z}\to\infty$}.
  \]
\end{thm}

Recall that for a normal matrix (i.e.\ which commutes with its adjoint), the
absolute value of the eigenvalues are equal to the singular values. Theorem
\ref{th:asydensity} reveals a striking contrast between $\mu_\alpha$ and
$\nu_{\alpha,0}$. The limiting law of the eigenvalues $\mu_{\alpha}$ has a
stretched exponential tail while the limiting law $\nu_{\alpha,0}$ of the
singular values is heavy tailed with power exponent $\alpha$, see e.g.\
\cite{BDG}. This does not contradict the identity
$\prod_{k=1}^n|\lambda_k(A)|= \prod_{k=1}^ns_k(A)$, but it does indicate that
$A$ is typically far from being a normal matrix. A similar shrinking
phenomenon appears already in the finite second moment case
\eqref{eq:QCLT}-\eqref{eq:CLT}: the law of the absolute value under the circular law
$\mathcal{U}_\sigma$ has density
\[
r\mapsto 2\sigma^{-2}r\ind_{[0,\sigma]}(r)
\]
in contrast with 
the density \eqref{eq:QCDensity} of the quarter-circular law
$\mathcal{Q}_{\sigma,0}$, even the supports differ by a factor $2$.

The proof of theorem \ref{th:mpz} is given in section \ref{ss:mpz}. It relies
on an extension to non-Hermitian matrices of the ``objective method'' approach
developed in \cite{bordenave-caputo-chafai-ii}. More precisely, we build an
explicit operator on Aldous' Poisson Weighted Infinite Tree (PWIT) and prove
that it is the local limit of the matrices $A_n$ in an appropriate sense.
While Poisson statistics arises naturally as in all heavy tailed phenomena,
the fact that a tree structure appears in the limit is roughly explained by
the observation that non vanishing entries of the rescaled matrix
$A_n=a_n^{-1}X$ can be viewed as the adjacency matrix of a sparse random graph
which locally looks like a tree. In particular, the convergence to PWIT is a
weighted-graph version of familiar results on the local structure of
Erd\H{o}s-R\'enyi random graphs.
%

The proof of theorem \ref{th:girko} is given in section \ref{se:girko}. It
relies on Girko's Hermitization method with logarithmic potentials, on theorem
\ref{th:mpz}, and on polynomial bounds on the extremal singular values needed
to establish a uniform integrability property. This extends the Hermitization
method to more general settings, by successfully mixing various arguments
already developed in \cite{bordenave-caputo-chafai-ii,cirmar,MR2722794}.
Following Tao and Vu, one of the key steps will be a lower bound on the
distance of a row of the matrix $A$ to a subspace of dimension at most $n -
n^{1- \gamma}$, for some small $\gamma >0$.

Girko's Hermitization method gives a characterization of $\mu_\alpha$ in terms
of its logarithmic potential (see appendix \ref{se:loppot}). In our settings,
however, this is not convenient to derive properties of the measure
$\mu_\alpha$, and our proof of theorem \ref{th:asydensity} is based on an
analysis of a self-adjoint operator on the PWIT and a recursive
characterization of the spectral measure from the resolvent of this operator.
This method is explained in section \ref{se:bipresolv} while the actual
computations on the PWIT are performed in section \ref{se:mu}.

Let us conclude with some final remarks. Following \cite{MR2549497}, the derivation of a Markovian version of theorems
\ref{th:mpz} and \ref{th:girko} is an interesting problem,
see \cite{MR2644041,bordenave-caputo-chafai-ii} for the symmetric case and
\cite{MR2575404,cirmar} for the light tailed non-symmetric case. In another
direction, it is also tempting to seek for an interpretation of
$\nu_{\alpha,z}$ and $\mu_\alpha$ in terms of a sort of graphical free
probability theory. Indeed, our random operators are defined on trees and tree
structures are closely related to freeness. Also, with a proper notion of
trace, it is possible to define the spectral measure of an operator, see e.g.\
\cite{brown86,haagerupschultz,lyons}. However these notions are usually
defined on algebras of bounded operators and we will not pursue this goal
here. Note finally that theorems \ref{th:mpz} and \ref{th:girko} remain
available for additive perturbations of finite rank, by following the
methodology used in \cite{MR2735731,MR2722794,tao-outliers}.

\subsection{Notation}

Throughout the paper, the notation $n\gg1$ means \emph{large enough $n$}. For
any $c \in [0,\infty]$ and any couple $f,g$ of positive functions defined in a
neighborhood of $c$ , we say that $f(t) \sim g(t)$ as $t$ goes to $c$, if
$\lim_{t\to c} f (t) / g(t) =1$. We denote by $\cD'(\dC)$ the set of
Schwartz-Sobolev distributions endowed with its usual convergence with respect
to all infinitely differentiable functions with bounded support $C^\infty_0(
\dC)$. We will consider the differential operators on $\dC \simeq \dR^2$, for
$z = x + i y$ (here $i=\sqrt{-1}$)
\[
\partial = \frac{1}{2} \PAR{\partial_x- i \partial_y}
\quad\text{and}\quad 
\bar \partial =\frac{1}{2}\PAR{\partial_x+ i \partial_y}.
\]
We have $ \partial \bar z = \bar \partial z = 0$, $\partial z = \bar \partial
\bar z = 1$ and the Laplace differential operator on $\mathbb{C}$ is given by
\[
\Delta = 4\partial  \bar \partial = \partial^2_x+\partial^2_y.
\] 
We use sometimes the shortened notation $A-z$ instead of $A-zI$.

\section{Bipartized resolvent matrix}
\label{se:bipresolv}

The aim of this section is to develop an efficient machinery to analyze the
spectral measures of a non-hermitian matrix which avoids a direct use of the logarithmic
potential and the singular values. Our approach builds upon similar methods in
the physics literature \cite{FZ97,Gudowska-Nowak,rogerscastillo,rogers}.

\subsection{Bipartization of a matrix}

Let $n$ be an integer, and $A$ be a $n \times n$ complex matrix. We introduce
the symmetrized version of $\nu_{A -z}$,
\[
\cnu_{A - z} 
= \frac{1}{2 n} \sum_{k= 1} ^n \left( \delta_{\sigma_k(A - z)} +  \delta_{- \sigma_k(A - z)} \right).
\]
Let $ \dC_+ = \{ z \in \dC : \Im(z) >0 \}$ and consider the quaternionic-type set
\[
\dH_+ = \left\{ U = \begin{pmatrix} \eta & z \\ \bar z & \eta\end{pmatrix} ,
  \eta \in \dC_+, z \in \dC \right\} \subset \cM_2 ( \dC).
\] 
For $z \in \dC, \eta \in \dC_+$ and $1 \leq i,j \leq n$ integers, we define
the elements of $\dH_+$ and $\cM_2 (\dC)$ respectively,
\[
U (z, \eta) = \begin{pmatrix} \eta & z \\ \bar z & \eta \end{pmatrix} \quad
\hbox{ and } \quad B_{ij} = \begin{pmatrix} 0 & A_{ij} \\ \bar A_{ji} &
  0 \end{pmatrix}.
\]
We define the matrix in $\cM_n ( \cM_2 ( \dC) ) \simeq \cM_{2n} ( \dC) $, $B =
(B_{ij})_{1 \leq i,j \leq n}$. Since $B_{ji} ^* = B_{ij}$, as an element of
$\cM_{2n} ( \dC) $, $B$ is an Hermitian matrix. Graphically, the matrix $A$
can be identified with an oriented graph on the vertex set $\{1, \cdots , n\}$
with weight on the oriented edge $(i,j)$ equal to $A_{ij}$. Then, the matrix
$B$ can be thought of as the bipartization of the matrix $A$, that is a
non-oriented graph on the vertex set $\{1,-1, \cdots, - n, n\}$, for every
integers $1 \leq i,j \leq n$ the weight on the non-oriented edge $\{ i , -
j\}$ is $A_{ij}$, and there is no edge between $i $ and $j$ or $-i$ and $-j$.

For $U \in \dH_+$, let $U \otimes I_n \in \cM_n ( \cM_2 ( \dC) )$ be the
matrix given by $( U \otimes I_n )_{ij} = \delta_{ij} U$, $1 \leq i,j \leq n$.
The resolvent matrix is defined in $\cM_n ( \cM_2 ( \dC) )$ by
\[
R(U)  = ( B - U \otimes I_n ) ^{-1},
\]
so that for all $1 \leq i,j\leq n$, $R(U)_{ij} \in \cM_2 ( \dC)$. For $1 \leq
k \leq n$, we write, with $U = U (z, \eta)$,
\begin{equation}\label{eq:RUabc}
  R (U)_{kk} = 
  \begin{pmatrix} 
    a_{k}(z,\eta) & b_{k} (z,\eta) \\   
    b'_{k} (z,\eta) &  c_{k} (z,\eta) 
  \end{pmatrix}.
\end{equation} 
The modulus of the entries of the matrix $R( U)_{kk}$ are bounded by $( \Im (
\eta) ) ^{-1}$ (see the forthcoming lemma \ref{le:Rbounded}).

As an element of
$\cM_{2n} ( \dC) $, $R$ is the usual resolvent of the matrix
\[
B(z) = B - U(z,0) \otimes I_n.
\] 
Indeed, with $U = U (z, \eta)$, 
\begin{equation}\label{eq:RU2n}
  R(U) =  ( B(z)  - \eta  I_{2n} ) ^{-1}.
\end{equation}
In the next proposition, we shall check that the eigenvalues of $B(z)$ are
$\pm \sigma_k ( A - z) $, $1 \leq k \leq n$, and consequently
\begin{equation}
  \label{eq:mubnz}
 \mu_{B(z)} = \cnu_{A -z}.
\end{equation}
It will follow that the spectral measures $\mu_A$ and $\cnu_{A-z}$ can be
easily recovered from the resolvent matrix. Recall that the Cauchy-Stieltjes
transform of a measure $\nu$ on $\dR$ is defined, for $\eta \in \dC_+$, as
\begin{equation*}
  m_{\nu} ( \eta) = \int_{\dR} \frac{ 1}{ x - \eta} \nu (dx).
\end{equation*}
The Cauchy-Stieltjes transform characterizes the measure. For a probability measure on $\dC$, it is possible to define a Cauchy-Stieltjes-like transform on quaternions, by setting for $U \in \dH_+$, 
\[
M_\mu ( U ) = \int_{\dC} \left(  \begin{pmatrix} 0 & \lambda \\ \bar \lambda & 0 \end{pmatrix} - U \right)^{-1} \mu ( d\lambda) \;  \in \; \dH_+. 
\]
This transform characterizes the measure : in $\cD' (\dC)$, $ \lim_{t \downarrow 0} (\partial M_\mu ( U ( z, it)  )_{12}  =  -  \pi \mu$. If $A$ is normal,  i.e. if $A^* A = A A^* $, then it can be checked that $R(U)_{kk} \in \dH_+$ and 
\[
\frac {1}{n } \sum_{k=1} ^n R(U)_{kk}  = M_{\mu_A}(U). 
\]
However, if $A$ is not normal, the above formula fails to hold and the next
proposition explains how to recover anyway $\mu_A$ from the resolvent.

\begin{thm}[From resolvent to spectral measure] \label{th:muabc} Let $U = U
  (z, \eta) \in \dH_+$, and $a_k,b_k,b'_k, c_k$ be as in \eqref{eq:RUabc}.
  Then \eqref{eq:mubnz} holds,
  \[
  m_{\cnu_{A - z}} (\eta) = \frac{1}{2n} \sum_{k=1} ^n \left(a_k (z,\eta) + c_k
  (z,\eta) \right),
  \]
  and, in $\cD' (\dC)$, 
  \[
  \mu_{A} = - \frac{1}{\pi n} \sum_{k=1} ^n \partial b_k (\cdot, 0) %
  = \lim_{t \downarrow 0} - \frac{1}{\pi n} \sum_{k=1} ^n \partial b_k
  (\cdot, i t ).
  \]
  In particular, if $A$ is a random matrix with exchangeable entries, then by
  linearity we get 
  \[
  m_{\dE \cnu_{A-z}}  (\eta) = \dE a_1 (z,\eta),
  \]
  and, in $\cD' (\dC)$, 
  \[
  \dE \mu_{A} %
  = - \frac{1}{\pi } \partial \dE b_1 (\cdot, 0) %
  = \lim_{t \downarrow 0} - \frac{1}{\pi } \partial \dE b_1 (\cdot, i t).
  \]
\end{thm}

\begin{proof}[Proof of theorem \ref{th:muabc}. ]
  Through a permutation of the entries, the matrix $B(z)$ is similar to
  \[
  \begin{pmatrix} 0 & ( A - z) \\ ( A - z) ^* & 0 \end{pmatrix},
  \]
  whose eigenvalues are easily seen to be $\pm \sigma_k ( A - z) $, $1 \leq k
  \leq n$. We get
  \[
  \tr R = \sum_{k=1}^n \left( a_k + c_k \right) %
  = \sum_{k=1} ^n ( \sigma_k (A - z) - \eta )^{-1} +   \sum_{k=1} ^n \ ( - \sigma_k (A - z) -
  \eta )^{-1}.
  \]
  And the first statement and \eqref{eq:mubnz} follow. Also, from
  \eqref{eq:LPVS}, in Appendix, for $z \notin \supp(\mu_{A})$,
  \begin{equation}\label{eq:PLH}
    U_{\mu_A} (z) %
    = \int \ln |�x|�    \mu_{B(z)} (dx) %
    = \frac{1}{2n} \ln  |�\det B (z)�|�, 
  \end{equation}
  where $U_\mu$ is the logarithmic potential of a measure $\mu$ on $\dC$, see
  \eqref{eq:logpot}. Recall that the differential of $X\mapsto\det(X)$ at
  point $X$ (invertible) in the direction $Y$ is $\tr ( X^{-1} Y ) \det(X)$
  (this is sometimes referred as the Jacobi formula). The sign of $\det B(z)$ is $(-1)^n$. We deduce that in $\cD'
  (\dC)$,
  \[
  \bar \partial  \ln  |�\det B(z) |� = \frac{ \bar \partial \det
    B(z) } { \det B(z) } = \tr \left\{ B(z) ^{-1}
    \bar \partial \begin{pmatrix} 0 & - z \\ - \bar z & 0 \end{pmatrix}
    \otimes I_n \right\}.\] With $R_{kk} = R( U (z,0))_{kk} = (
  B(z)^{-1})_{kk}$, we get from $\bar \partial z = 0$, $\bar \partial \bar z =
  1$,
  \[
  \bar \partial   \ln  |\det B(z) | = \sum_{k=1}^n \tr\left\{
    R_{kk} \begin{pmatrix} 0 & 0 \\ -1 & 0 \end{pmatrix}\right\} = -
  \sum_{k=1}^n b_{k} (z, 0).
  \]
  Now from Equation \eqref{eq:lap}, in $\cD'(\dC)$, using 
  $\Delta=4\partial\bar\partial$,
  \[
  2\pi \mu_A = \Delta U_{\mu_A} 
  = \Delta\frac{1}{2n}\ln |\det B(z) |
  =- \frac {2}{n}\sum_{k=1}^n \partial b_{k}.
  \]
  To get the limit as $t \downarrow 0$, we note that for real $t >0$,
  \[
  \int \ln  | x    -  i t |  \mu_{B(z) } (dx) %
  = \frac{1}{2 n} \ln | \det (B (z) - i t ) |  .
  \]
Note that $ \det (B (z) - i t )$ is real and its sign is $(-1)^n$. As $t \downarrow 0$, the left hand side of the above identity converges in
  $\cD'(\dC)$, to $U_{\mu_A}$. Taking the Laplacian, and arguing as above, we
  get
  \begin{equation}\label{eq:Deltat}
 \Delta \int \ln  | x  - i  t|    \mu_{B(z) } (dx) %
    = - \frac{2}{n}  \sum_{k=1} ^n   \partial  b_k (z, i t ).
  \end{equation}
  The conclusion follows.
\end{proof}

Note that even if $- \sum_k \partial b_k$ is a measure on $\dC$, for each
$1\leq k \leq n$, $ - \partial b_k $ is not in general a measure on $\dC$
(default of positivity, this can be checked on $2\times 2$ matrices).

\subsection{Bipartization of an operator}

We shall generalize the above finite dimensional construction. Let $V$ be a countable set and let
$\ell^2(V)$ denote the Hilbert space defined by the scalar product
\[
\langle \phi,\psi \rangle %
:= \sum_{u\in V} \bar\phi_u\psi_u\,,\quad\quad \phi_u %
= \langle \d_u,\phi \rangle,
\] 
where $\delta_u$ is the unit vector supported on $u \in V$. Let $\cD(V)$
denote the dense subset of $\ell^2 (V)$ of vectors with finite support. Let
$(w_{uv})_{u, v \in V}$ be a collection of complex numbers such that for all
$u \in V$,
\[
\sum_{v \in V} |w_{uv}|^2  +  |w_{vu}|^2 < \infty 
\]
We may then define a linear operator $A$ on $\cD(V)$, by the formula, 
\begin{equation}\label{eq:defopA}
\langle  \delta_u  , A  \delta_v  \rangle = w_{uv}.
\end{equation}

Let $\hat V$ be a set in bijection with $V$, the image of $v \in V$ being
denoted by $\hat v \in \hat V$. We set $V^b = V \cup \hat V$ and define the symmetric 
operator $B$ on $\cD(V^b)$, by the formulas, 
\begin{eqnarray}
  \langle \delta_u ,  B  \delta_{\hat v} \rangle   =  \overline{ \langle  \delta_{\hat v}  ,  B  \delta_{u} \rangle }  =&   w_{uv}   \nonumber \\
  \langle \delta_u ,  B  \delta_{ v} \rangle    =  \langle \delta_{\hat u} ,  B  \delta_{ \hat v} \rangle   = & 0  \label{eq:defopB}.
\end{eqnarray}
In other words, if $\Pi_u : \ell^2 ( V^b) \to \dC^2$ denotes the orthogonal
projection on $(u,\hat u)$,
\[ 
\Pi_u B \Pi^*_v %
=   \begin{pmatrix} 0  & w_{uv}  \\   \bar w_{vu}  &    0  \end{pmatrix}.
\]
For $z \in \dC$, we also define on $\cD(V^b)$, the symmetric operator $B(z)$: for all
$u,v$ in $V$,
\begin{eqnarray*}
  \langle \delta_u ,  B(z)  \delta_{\hat v} \rangle   =  \overline{ \langle  \delta_{\hat v}  ,  B (z) \delta_{u} \rangle }  =&   w_{uv}   - z \ind( u = v)  \\
  \langle \delta_u ,  B (z)  \delta_{ v} \rangle    =  \langle \delta_{\hat u}  ,  B (z) \delta_{\hat v} \rangle   = & 0  .
\end{eqnarray*}
Hence, if we identify $V^b$ with $\{1, 2\} \times V$, we have
\begin{equation}\label{eq:defopBz}
  B(z) = B - U (z,0) \otimes I_V.
\end{equation}
The operator $B(z)$ is symmetric and it has a closure on a domain $D(B)
\subset \ell^2 (V^b)$. We also denote by $B(z)$ the closure of $B(z)$. If $B$ is
self-adjoint then $B(z)$ is also self-adjoint (recall that the sum of a
bounded self-adjoint operator and a self-adjoint operator is also a
self-adjoint operator). Recall also that the spectrum of a self--adjoint
operator is real. For all $U = U(z,\eta) \in \dH_+$, $B(z) - \eta I_{V^b} = B
- U (z,\eta) \otimes I_V$ is invertible with bounded inverse and the resolvent
operator is then well defined by
\[
R(U) = ( B(z) - \eta I_{V^b}) ^{-1}. 
\]
We may then define 
\[
R (U)_{vv} %
= \Pi_v R(U) \Pi^*_v %
= 
\begin{pmatrix} 
  a_{v}(z,\eta) & b_{v} (z,\eta) \\
  b'_{v} (z,\eta) & c_{v} (z,\eta)
\end{pmatrix}.
\]
In the sequel, we shall use some properties of resolvent operators.
\begin{lem}[Properties of resolvent] \label{le:Rbounded} Let $B$ be the above
  bipartized operator. Assume that $B$ is self-adjoint and let $U = U (z,
  \eta) \in \dH_+$, $v \in V$. Then, $a_{v}, c_v \in \dC_+$, for each $z \in
  \dC$, the functions $a_v(z,\cdot), b_v(z, \cdot) , b'_v (z, \cdot),
  c_v(z,\cdot)$ are analytic on $\dC_+$, and
  \[
  | a_v | \leq (\Im(\eta ) ) ^{-1}, \quad | c_v | \leq (\Im( \eta ) ) ^{-1},
  \quad | b_v | \leq (2\Im (\eta) ) ^{-1} \quad \hbox{and } \quad | b'_v |
  \leq (2\Im( \eta ) ) ^{-1}.
  \]
  Moreover, if $\eta \in i \dR_+$, then $a_v$ and $c_v$ are pure imaginary and
  $b'_v = \bar b_v$.
\end{lem}

\begin{proof}
 For a proof of the first statements refer e.g. to Reed and Simon \cite{reedsimon}.   For the last statement concerning $\eta \in i \dR_+$, we define the
  skeleton of $B(z)$ as the graph on $V^b$ obtained by putting an edge between
  two vertices $u,v$ in $V^b$, if $\langle \delta_u , B(z) \delta_v \rangle
  \ne 0$. Then since there is no edge between two vertices of $V$ or $\hat V$,
  the skeleton of $B(z)$ is a bipartite graph.

Assume first that $B(z)$ is bounded:
  for all $u \in V^b$, $\| B(z) \delta_{u}\| \leq C$. Then for $|\eta | > C$,
  the series expansion of the resolvent gives
  \[
  R(U) = - \sum_{n = 0} ^\infty \frac{ B(z) ^{n} }{\eta^{n+1}}.
  \]
  However since the skeleton is a bipartite graph, all cycles have an even
  length. It implies that for $n$ odd, $\langle \delta_{u} ,B(z)
  ^{n}\delta_{u} \rangle = 0$. Applied first to $v \in V$, we deduce that for $|\eta| >
  C$, $a(z,-\bar \eta) = -\bar a (z,\eta)$ and then applied to $\hat v$, we get
  $c(z,-\bar \eta) = -\bar c (z,\eta)$. We may then extend to $\dC_+$ this
  last identity by analyticity. For $\eta = i t \in i \dR_+$, we deduce that
  $a_v$ and $c_v$ are pure imaginary. Similarly, since the skeleton is a
  bipartite graph, a path from a vertex $v \in V$ to a vertex $\hat u \in \hat
  V$ must of be of odd length. We get for $|\eta| > C$
  \begin{eqnarray*}
    \bar b'_v ( z ,  - \bar \eta ) = \overline{ \langle \delta_{\hat v} , R(U(z,-\bar\eta)) \delta_{v} \rangle}  & = &  - \sum_{n = 0} ^\infty    \frac{  \overline{ \langle  \delta_{\hat v} ,  B(z) ^{2n+1} \delta_{v} \rangle } }{  \eta ^{2n+2}}  \\
    & = &  \langle \delta_{ v} , R(U) \delta_{\hat v} \rangle  = b_v (z, \eta),
  \end{eqnarray*}
  where we have used the symmetry of $B(z)$. It follows that $b'_v ( z , -
  \bar \eta ) = \bar b_v (z, \eta )$. If $B(z)$ is not bounded, then $B(z)$ is
  limit of a sequence of bounded operators and we conclude by invoking Theorem
  VIII.25(a) in \cite{reedsimon}.
\end{proof}

\subsection{Operator on a tree}
\label{subsec:optree}

We keep the setting of the above paragraph and consider a (non-oriented) tree
$T = (V,E)$ on the vertices $V$ with edge set $E$ (recall that a tree is a
connected graph without cycles). For ease of notation, we note $u \sim v$ if
$\{u,v\} \in E$. We assume that if $\{u,v\} \notin E$ then $w_{uv} = w_{vu} =
0$. In particular $w_{vv} = 0$ for all $v \in V$. We continue to consider the
operator $A$ defined by \eqref{eq:defopA}.

In the special case when $w_{uv} = \overline w_{vu}$ for all $u, v$ in $V$,
the operator $A$ is symmetric and we first look for sufficient conditions for
$A$ to be essentially self-adjoint.
\begin{lem}[Criterion of self-adjointness]\label{le:criteresa}
  Let $\kappa > 0$ and $T = (V,E)$ be a tree. Assume that for all $u, v \in
  V$, $w_{uv} = \overline w_{vu}$ and that if $\{u,v\} \notin E$ then $w_{uv}
  = w_{vu} = 0$. Assume also that there exists a sequence of connected finite
  subsets $(S_n)_{n \geq 1}$ in $V$, such that $ S_n \subset S_{n+1}$, $\cup_n
  S_n = V$, and for every $n$ and $v \in S_n$,
  \[
  \sum_{u \notin S_n : u \sim v } |w_{uv} |^2\leq \kappa.
  \] 
  Then $A$ is essentially self-adjoint.
\end{lem}
For a proof, see \cite[Lemma A.3]{bordenave-caputo-chafai-ii}. The above lemma
has an interesting corollary for the bipartized operator $B$ of $A$ defined by
\eqref{eq:defopB}-\eqref{eq:defopBz}.

\begin{cor}[Criterion of self-adjointness of bipartized
  operator]\label{cor:criteresa}
  Let $\kappa > 0$ and $T = (V,E)$ be a tree. Assume that if $\{u,v\} \notin
  E$ then $w_{uv} = w_{vu} = 0$. Assume also that there exists a sequence of
  connected finite subsets $(S_n)_{n \geq 1}$ in $V$, such that $ S_n \subset
  S_{n+1}$, $\cup_n S_n = V$, and for every $n$ and $v \in S_n$,
  \[
  \sum_{u \notin S_n : u \sim v } \left( |w_{uv} |^2 + |w_{vu} |^2 \right)  \leq \kappa.
  \]  
  Then for all $z \in \dC$, $B(z)$ is self-adjoint.
\end{cor}

\begin{proof}
  From \eqref{eq:defopBz}, it is sufficient to check that $B$ is self-adjoint.
  Let $\o \in V$ be a distinguished vertex, we define two disjoint trees
  $G_{\o} = (V_{\o}, E_{\o}) $ and $\hat G_{\o} = (\hat V_{\o}, \hat E_{\o}) $
  on a partition $(V_{\o}, \hat V_{\o})$ of $V^b$ as follows. The trees
  $G_{\o}$ and $\hat G_{\o}$ are the unique trees such that $\o \in V_{\o}$,
  $\hat \o \in\hat V_{\o}$ and that satisfy the following properties
  \begin{itemize}
  \item[(i)] if $\{u,v\} \in E$ and $u$ in $V_{\o}$ (or $\hat V_{\o}$) then
    $\hat v \in V_{\o}$ (or $\hat V_{\o}$) and $\{u,\hat v\} \in E_{\o}$ (or
    $\hat E_{\o}$),
  \item[(ii)] if $\{u,v\} \in E$ and $\hat u$ in $V_{\o}$ (or $\hat V_{\o}$)
    then $ v \in V_{\o}$ (or $\hat V_{\o}$) and $\{\hat u,v\} \in E_{\o}$ (or
    $\hat E_{\o}$).
  \end{itemize}
  We note that by construction if $u \in V_{\o}$ and $v \in \hat V_{\o}$ then
  $\langle \delta_u , B \delta_{ v} \rangle =0$. If follows that the operator
  $B$ decomposes orthogonally into two operators $B_{\o}$ and $\hat B_{\o}$ on
  domains in $\ell^2(V_{\o})$ and $\ell^2(\hat V_{\o})$ respectively: $B =
  B_{\o} \oplus \hat B_{\o}$. We may then safely apply lemma
  \ref{le:criteresa} to $B_{\o}$ and $\hat B_{\o}$.
\end{proof}

When the operator $B$ is self-adjoint, the resolvent operator has a nice
recursive expression due to the tree structure. Let $\o \in V$ be a
distinguished vertex of $V$ (in graph language, we root the tree $T$ at $\o$).
For each $v \in V \backslash \{\o\}$, we define $V_v \subset V$ as the set of
vertices whose unique path to the root $\o$ contains $v$. We define $T_v =
(V_v, E_v)$ as the subtree of $T$ spanned by $V_v$. We finally consider $A_v$,
the projection of $A$ on $V_v$, and $B_v$ the bipartized operator of $A_v$.
The skeleton of $A_v$ is contained in $T_v$. Finally, we note that if $B$ is
self-adjoint then so is $B_v (z)$ for every $z \in \dC$. The next lemma can be
interpreted as a Schur complement formula on trees.

\begin{lem}[Resolvent on a tree] \label{le:schurB} Assume that $B$ is
  self-adjoint and let $U = U(z,\eta) \in \dH_+$. Then
  \begin{eqnarray*}
    R(U) _{\o\o}   &=&  - \left( U  + \sum_{v \sim \o}   
      \begin{pmatrix}    
        0      &  w_{ \o v}    \\ 
        \overline w_{v \o  }  & 0   
      \end{pmatrix}  \wt R (U)_{vv}   
      \begin{pmatrix}    0      &  w_{v \o}    \\ 
        \overline w_{\o v }  & 0   
      \end{pmatrix} \right)^{-1} ,
  \end{eqnarray*}
  where $ \wt R (U)_{vv} = \Pi_v R_{B_v} (U) \Pi^*_v$ and $R_{B_v} (U) =
  (B_v(z) - \eta ) ^{-1}$ is the resolvent operator of $B_v$.
\end{lem}

\begin{proof}
  Define the operator $C$ on $\cD(V^b)$ by its matrix elements 
  \[
  C_{\o} %
  := \Pi_{\o} C \Pi^*_{\o} %
  = - U( z , 0) \; , %
  \quad \quad %
  C_v := \Pi_{\o} C \Pi^*_v= \Pi_{v} C \Pi^*_{\o}  %
  = %
  \begin{pmatrix} 0 & w_{ \o v} \\
    \overline w_{v \o } & 0 %
  \end{pmatrix}
  \] 
  for all $v \in V$ such that $v \sim \o$, and $ \Pi_u C \Pi^*_v =0$
  otherwise. The operator $C$ is symmetric and bounded. Its extension to
  $\ell^2 (V^b)$ is thus self-adjoint (also denoted by $C$). In this way, we
  have from $ V = \{ \o \}\bigcup_{v \sim \o } V_v$,
  \[
  B (z ) =   C + \widetilde B \quad \text{ with }\quad  \widetilde B  = \bigoplus_{v \sim \o }    B_v(z) .
  \]
  We shall write $\wt R(U) = (\wt B - \eta I ) ^{-1} $ for the associated
  resolvent of $\wt B$. From the resolvent identity, these operators satisfy
  \begin{equation} \label{resid}
    \wt R(U)  C R(U) = \wt R(U) - R(U)\,.
  \end{equation}
  Set $\wt R_{u v}= \Pi_u \wt R (U) \Pi ^*_v$ and $R_{u v}=\Pi_u R (U)
  \Pi^*_v$. Observe that $\wt R_{\o \o}= - \eta^{-1} I_2 $. Also the direct
  sum decomposition $ V = \{ \o \}\bigcup_{v \sim \o } V_v$ implies $\wt R_{v
    v} = \Pi_v R_{B_v} (U) \Pi^*_v$ and $\wt R_{u v} = 0$ for every $ u\neq v$
  with $u \sim \o$, $v \sim \o$. Similarly we have that $\wt R_{\o v} = 0 =
  \wt R_{v \o}$ for every $v \in V \backslash \{ \o\}$. Using the identity
  $\sum_{u \in V} \Pi_u^* \Pi_u = I$, we get
  \begin{eqnarray*}
  \Pi_{\o} \wt R(U) C R(U)\Pi^*_{\o} & =&  \wt R_{\o\o } C_{\o}  R_{\o\o} + \sum_{ v \sim  \o} \wt R_{\o\o}C_v R_{v\o} \\
  &= &\eta^{-1} U(z,0) R_{\o \o} - \eta^{-1} \sum_{ v\sim \o} C_v R_{v\o}.
  \end{eqnarray*} 
  We compose the identity \eqref{resid} on the left by $\Pi_v$ and on the
  right by $\Pi_{\o}^*$, we obtain, for $v \sim \o$,
  \[
  \wt R_{v v}  C^*_v R_{\o \o} = - R_{v \o} \,. 
  \]
  We finally compose \eqref{resid} on the left by $\Pi_{\o}$ and on the right
  by $\Pi_{\o}^*$,
  \[
  \eta^{-1} U(z,0) R_{\o \o} + \eta ^{-1} \sum_{v \sim \o} C_v \wt R_{v v}
  C^*_v R_{\o \o}\, = - \eta^{-1} I_2 - R_{\o \o},
  \]
  or equivalently $( U(z,\eta) + \sum_{v \sim \o} C_v \wt R_{v v} C^*_v )
  R_{\o \o}\, = - I_2$.
\end{proof}

\subsection{Local operator convergence}

In the next paragraphs, we are going to prove that the sequence of random
matrices $(A_n)$ converges to a limit random operator on an infinite tree. Let
us recall a notion of convergence that we have already used in
\cite{bordenave-caputo-chafai-ii}.

\begin{defi}[Local convergence]\label{def:convloc}
  Suppose $(A_n)$ is a sequence of bounded operators on $\ell^2(V)$ and $A$ is
  a linear operator on $\ell^2(V)$ with domain $D(A)\supset \cD(V)$. For any
  $u,v\in V$ we say that $(A_n,u)$ converges locally to $(A,v)$, and write
  \[
  (A_n,u) \to (A,v)\,,
  \] 
  if there exists a sequence of bijections $\sigma_n:V\to V$ such that
  $\sigma_n (v) = u$ and, for all $\phi\in\cD(V)$,
  \[
  \sigma_n ^{-1} A_n \sigma_n \phi \to A \phi\,, 
  \]
  in $\ell^2(V)$, as $n\to\infty$. 
\end{defi}

Assume in addition that $A$ is closed and $\cD(V)$ is a core for $A$
(i.e.\ the closure of $A$ restricted to $\cD(V)$ equals $A$). Then, the local
convergence is the standard strong convergence of operators up to a
re-indexing of $V$ which preserves a distinguished element. With a slight
abuse of notation we have used the same symbol $\sigma_n$ for the linear
isometry $\si_n: \ell^2(V)\to \ell^2(V)$ induced in the obvious way. As
pointed out in \cite{bordenave-caputo-chafai-ii}, the point for using
Definition \ref{def:convloc} lies in the following theorem on strong resolvent
convergence.

\begin{thm}[From local convergence to resolvents]\label{th:strongres} 
  Assume that $(A_n)$ and $A$ satisfy the conditions of Definition
  \ref{def:convloc} and $(A_n,u) \to (A,v)$ for some $u,v\in V$. Let $B_n$ be
  the self-adjoint bipartized operator of $A_n$. If the bipartized operator
  $B$ of $A$ is self-adjoint and $\cD(V^b)$ is a core for $B$, then, for all
  $U \in \dH_+$,
  \begin{equation}\label{strresconv}
    R_{B_n} ( U) _{uu}  \to R_{B} ( U) _{vv}.
  \end{equation}
  where $ R_{B} ( U) _{vv} = \Pi_v R_B (U)\Pi_v ^*$ and $ R_B (U) = ( B(z) -
  \eta ) ^{-1}$ is the resolvent of $B(z)$.
\end{thm}

\begin{proof}[Proof of theorem \ref{th:strongres}]
  It is a special case of Reed and Simon \cite[Theorem VIII.25(a)]{reedsimon}.
  Indeed, we first fix $z \in \dC$ and extend the bijection $\sigma_n$ to
  $V^b$ by the formula, for all $w \in V$, $\sigma_n ( \hat w ) = \hat
  \sigma_n (w) $. Then we define $\wt B_n (z) = \sigma_n ^{-1} B_n (z)
  \sigma_n$, so that $\wt B_n (z) \phi \to B (z) \phi$ for all $\phi$ in a
  common core of the self--adjoint operators $\wt B_n(z) , B(z) $. This
  implies the strong resolvent convergence, i.e.\ $(\wt B_n (z) - \eta
  I)^{-1}\psi \to (B(z) - \eta I)^{-1}\psi$ for any $\eta \in\dC_+$, $\psi\in
  \ell^2(V)$. We conclude by using the identities : $\Pi_v (\wt B_n (z) - \eta
  I)^{-1}\d_v = \Pi_u (B_n(z) - \eta I)^{-1}\d_u$ and $\Pi_v (\wt B_n (z) -
  \eta I)^{-1} \d_{\hat v} = \Pi_u (B_n(z) - \eta I)^{-1}\d_{\hat u}$.
\end{proof}

We shall apply the above theorem in cases where the operators $A_n$ and $A$ are
random operators on $\ell^2(V)$, which satisfy with probability one the
conditions of theorem \ref{th:strongres}. In this case we say that $(A_n ,u)
\to (A ,v)$ \emph{in distribution} if there exists a random bijection
$\sigma_n$ as in Definition \ref{def:convloc} such that $\sigma_n ^{-1} A_n
\sigma_n \phi $ converges in distribution to $A \phi $, for all $\phi \in
\cD(V)$ (where a random vector $\psi_n \in \ell^2 (V)$ converges in
distribution to $\psi$ if $\lim_{n\to\infty} \dE f (\psi_n) = \dE f(\psi)$ for
all bounded continuous functions $f:\ell^2 (V)\to\dR$). Under these
assumptions then \eqref{strresconv} becomes convergence in distribution of
(bounded) complex random variables. Note that in order to prove theorems \ref{th:mpz}, \ref{th:girko}, we will also need almost-sure convergence statements. 

\subsection{Poisson Weighted Infinite Tree  (PWIT)}
\label{subseq:PWIT}

We now define an operator on an infinite rooted tree with random
edge--weights, the Poisson weighted infinite tree (PWIT) introduced by Aldous
\cite{aldous92}, see also \cite{aldoussteele}.

Let $\r$ be a positive Radon measure on $\dR$ such that $\r (\dR) = \infty$.
$\pwit(\r)$ is the random weighted rooted tree defined as follows. The vertex
set of the tree is identified with $ \dN^f:= \cup_{k \in \dN} \dN^k$ by
indexing the root as $\dN^0 = \o$, the offsprings of the root as $\dN$ and,
more generally, the offsprings of some $v \in \dN^k$ as $(v1),(v2), \cdots \in
\dN^{k+1}$ (for short notation, we write $(v1)$ in place of $(v,1)$). In this
way the set of $v\in\dN^n$ identifies the $n^\text{th}$ generation. We then
define $T$ as the tree on $\dN^f$ with (non-oriented) edges between the
offsprings and their parents.

We denote by $\mathrm{Be}(1/2)$ the Bernoulli probability distribution
$\frac 1 2 \delta_0 + \frac 1 2\delta_1$. Now assign marks to the edges of the
tree $T$ according to a collection $\{ \Xi_v \}_{v \in \dN^f}$ of independent
realizations of the Poisson point process with intensity measure $\r \otimes
\mathrm{Be}(1/2)$ on $\dR \times \{0,1\}$. Namely, starting from the root
${\o}$, let $ \Xi_{\o} = \{(y_1,\e_1),(y_2,\e_2),\dots\}$ be ordered in such a
way that $|y_1| \leq |y_2| \leq \cdots$, and assign the mark $(y_i,\e_i)$ to
the offspring of the root labeled $i$. Now, recursively, at each vertex $v$ of
generation $k$, assign the mark $(y_{v i}, \e_{vi}) $ to the offspring labeled
$v i$, where $\Xi_{v}=\{( y_{v 1},\e_{v1}), (y_{v 2}, \e_{v2}), \dots \}$
satisfy $|y_{v 1}| \leq |y_{v 2}| \leq \cdots$. The Bernoulli mark $\e_{vi}$
should be understood as an orientation of the edge $\{v, vi\}$ : if $\e_{vi} =
1$, the edge is oriented from $vi$ to $v$ and from $v$ to $vi$ otherwise.

For a probability measure $\theta$ on $S^1$, we introduce the measure on
$\dC$, for all Borel $D$:
\begin{equation}
\label{eq:elltheta}
\ell_\theta (D)  = \int_0 ^\infty \hspace{-3pt} \int_{S^1} \ind_{\{\omega^{-\alpha} r \in D\}} \theta ( d \omega ) dr 
\end{equation}
Consider a realization of $\pwit(2 \ell_\theta)$. We now define a random
operator $A$ on $\cD (\dN^f)$ by the formula, for all $v \in \dN^f$ and $k \in
\dN$,
\begin{equation}\label{eq:defA}
\langle \delta_{v} ,  A \delta_{v k  }  \rangle 
= \e_{v k} y_{v k} ^{-1/\alpha}   \quad \hbox{and } \quad \langle \delta_{v k } ,  A \delta_{v  }  \rangle = (1- \e_{v k} ) y_{v k}^{-1/\alpha}  
\end{equation}
and $\langle \delta_{ v} , A \delta_{u }   \rangle = 0$ otherwise.  It is an operator as in \S \ref{subsec:optree}. Indeed, if $u = vk $ is an offspring of $v$, we set
\begin{equation}\label{eq:wApwit}
  w_{vu} = \e_{v k} y_{v k} ^{-1/\alpha}  \quad \hbox{and } \quad w_{uv} =(1-\e_{v k} )y_{v k} ^{-1/\alpha},
\end{equation}
otherwise, we set $w_{uv} = 0$. We may thus consider the bipartized operator
$B$ of $A$.

\begin{prop}[Self-adjointness of bipartized operator on PWIT] \label{prop:sa}
  Let $A$ be the random operator defined by \eqref{eq:defA}. With
  probability one, for all $z \in \dC$, $B(z)$ is self-adjoint.
\end{prop}

We shall use Corollary \ref{cor:criteresa}. We start with a technical lemma
proved in \cite[Lemma A.4]{bordenave-caputo-chafai-ii}.

\begin{lem}\label{le:taufinite}
  Let $\kappa >0$, $0 < \alpha < 2$ and let $0 < x_1 < x_2 < \cdots$ be a
  Poisson process of intensity $1$ on $\dR_+$. Define $\tau = \inf \{ t \in
  \dN : \sum_{k = t+1}^\infty x_k^{-2/\alpha} \leq \kappa\}$. Then $\dE \tau $
  is finite and goes to $0$ as $\kappa$ goes to infinity.
\end{lem}

\begin{proof}[Proof of proposition \ref{prop:sa}]
  For $\kappa >0$ and $v \in \dN^f$, we define 
  \[
  \tau_v= \inf \{ t \geq 0 : \sum_{k = t+1}^\infty |y_{v k}|^{-2/\alpha} \leq
  \kappa\}.
  \]
  The variables $(\tau_{v})$ are i.i.d.\ and by lemma \ref{le:taufinite},
  there exists $\kappa >0$ such that $\dE \tau_{v} < 1$. We fix such $\kappa$.
  Now, we put a green color to all vertices $v$ such that $\tau_v\geq 1$ and a
  red color otherwise. We consider an exploration procedure starting from the
  root which stops at red vertices and goes on at green vertices. More
  formally, define the sub-forest $T^g$ of $T$ where we put an edge between
  $v$ and $vk $ if $v$ is a green vertex and $1 \leq k \leq \tau_{v}$. Then,
  if the root $\o$ is red, we set $S_1 = C^g (T) = \{\o\}$. Otherwise, the
  root is green, and we consider $ T^g_{\o} = ( V^g_{\o} , E^g_{\o}) $ the
  subtree of $T^g$ that contains the root. It is a Galton-Watson tree with
  offspring distribution $\tau_{\o}$. Thanks to our choice of $\kappa$, $
  T^g_{\o} $ is almost surely finite. Consider $L^g_{\o}$ the leaves of this
  tree (i.e. the set of vertices $v$ in $ V^g_{\o} $ such that for all $1 \leq
  k \leq \tau_v$, $vk$ is red). We set $S_1 = V^g_{\o} \bigcup_{v \in
    L^g_{\o}} \{1 \leq k \leq \tau_v : v k \}$. Clearly, the set $S_1$
  satisfies the condition of Lemma \ref{le:criteresa}.

  Now, we define the outer boundary of $\{\o\}$ as $\partial_\tau \{\o \}=
  \{1, \cdots, \tau_{\o}\}$ and for $v= (i_1, \cdots i_k) \in \dN^f \backslash
  \{\o\}$ we set $\partial_\tau \{v \} = \{(i_1,\cdots, i_{k-1}, i_{k} +1) \}
  \cup \{(i_1,\cdots, i_{k},1), \cdots, (i_1,\cdots, i_{k},\tau_v)\} $. For a
  connected set $S$, its outer boundary is
  \[
  \partial_\tau S %
  = \left( \bigcup_{ v \in S} \partial_\tau \{v\} \right) \backslash S.
  \]
  Now, for each vertex $u_1, \cdots, u_k \in \partial_\tau S_1 $, we repeat
  the above procedure to the rooted subtrees $T_{u_1}, \cdots, T_{u_k}$. We
  set $S_2 = S_1 \bigcup \cup_{1 \leq i \leq k} C^b ( T_{u_i})$. Iteratively,
  we may thus almost surely define an increasing connected sequence $(S_n)$ of
  vertices with the properties required for Corollary \ref{cor:criteresa}.
\end{proof} 

\subsection{Local convergence to PWIT}

We may now come back to the random matrix $A_n$ defined by
\eqref{eq:defAn}. We extend it as an operator on $\cD(\dN^f)$ by setting for
$1 \leq i , j \leq n $, $\langle \delta_{i} , A \delta_{ j } \rangle =
A_{i,j}$ and otherwise, if either $i$ or $j$ is in $\dN^f \backslash\{1,
\cdots n\}$, $\langle \delta_{i} , A \delta_{ j } \rangle = 0 $.

The aim of this paragraph is to prove the following theorem.

\begin{thm}[Local convergence to PWIT]\label{th:locconvPWIT} 
  Assume (H1). Let $A_n$ be as above and $A$ be the operator associated to
  $\pwit(2\ell_\theta)$ defined by \eqref{eq:defA}. Then in distribution
  $(A_n,1) \to (A,\o)$.
\end{thm}

Up to small differences, this theorem has already been proved in \cite[Section
2]{bordenave-caputo-chafai-ii}. We review here the method of proof and stress
the differences. The method relies on the \emph{local weak convergence}, a
notion introduced by Benjamini and Schramm \cite{benjaminischramm}, Aldous and
Steele \cite{aldoussteele}, see also Aldous and Lyons \cite{aldouslyons}.

We define a network as a graph with weights on its edges taking values in some
metric space. Let $G_n$ be the complete network on $\{1,\ldots,n\}$ whose
weight on edge $\{i,j\}$ equals $(\xi^n_{i,j})$, for some collection
$(\xi^n_{i , j})_{1 \leq i \leq j \leq n}$ of i.i.d.\ complex random
variables. We set $\xi^n _ {j,i} = \xi^n_{i,j}$. We consider the rooted
network $(G_n,1)$ obtained by distinguishing the vertex labeled $1$.

We follow Aldous \cite[Section 3]{aldous92}. For every fixed realization of
the marks $(\xi^n_{i j})$, and for any $B,H\in\dN$, such that $(B^{H+1} -
1)/(B-1) \leq n$, we define a finite rooted subnetwork $(G_n,1)^{B,H}$ of
$(G_n,1)$, whose vertex set coincides with a $B$--ary tree of depth $H$ with
root at $1$. To this end we partially index the vertices of $(G_n,1)$ as
elements in
\[
J_{B,H} = \cup_{\ell=0}^H \{1,\cdots, B\}^\ell \subset \dN^f,
\]
the indexing being given by an injective map $\sigma_n$ from $J_{B,H}$ to
$V_n:=\{1,\dots,n\}$. We set $I_{\o} = \{ 1 \}$ and the index of the root $1$
is $ \sigma_n^{-1} (1) = {\o}$. The vertex $v\in V_n \backslash I_{{\o}}$ is
given the index $(k) = \sigma_n^{-1} (v)$, $1 \leq k\leq B$, if $\xi^n_{1,v}$
has the $k^\text{th}$ smallest absolute value among $\{\xi^n_{1,j}, \,j\neq
1\}$, the marks of edges emanating from the root $1$. We break ties by using
the lexicographic order. This defines the first generation. Now let $I_1$ be
the union of $I_{\o}$ and the $B$ vertices that have been selected. If $H\geq
2$, we repeat the indexing procedure for the vertex indexed by $(1)$ (the
first child) on the set $V_n \backslash I_1$. We obtain a new set
$\{11,\cdots,1B\}$ of vertices sorted by their weights as before (for short
notation, we concatenate the vector $(1,1)$ into $11$). Then we define $I_{2}$
as the union of $I_1$ and this new collection. We repeat the procedure for
$(2)$ on $V_n \backslash I_{2}$ and obtain a new set $\{21,\cdots,2 B\}$, and
so on. When we have constructed $\{B1,\cdots,BB\}$, we have finished the
second generation (depth $2$) and we have indexed $(B^{3} - 1)/(B-1)$
vertices. The indexing procedure is then repeated until depth $H$ so that
$(B^{H+1} - 1)/(B-1)$ vertices are sorted. Call this set of vertices
$V_n^{B,H} = \sigma_n J_{B,H} $. The subnetwork of $G_n$ generated by
$V_n^{B,H}$ is denoted $(G_n,1)^{B,H}$ (it can be identified with the original
network $G_n$ where any edge $e$ touching the complement of $V_n^{B,H}$ is
given a mark $x_e=\infty$). In $(G_n,1)^{B,H}$, the set $\{u1,\cdots,u B \}$
is called the set of offsprings of the vertex $u$. Note that while the vertex
set has been given a tree structure, $(G_n,1)^{B,H}$ is still a complete
network on $V_n^{B,H}$. The next proposition shows that it nevertheless
converges to a tree (i.e. extra marks diverge to $\infty$) if the $\xi^n_\ij$
satisfy a suitable scaling assumption.

Let $\rho$ be a Radon measure on $\dC$ and let $T$ be a realization of
$\pwit(\rho)$ defined in \S \ref{subseq:PWIT}. For the moment, we remove the
Bernoulli marks $(\e_v)_{v \in \dN^f}$ and, for $v \in \dN^f$ and $k \in \dN$,
we define the weight on edge $\{v,vk\}$ to simply be $y_{vk}$. Then $(T,\o)$
is a rooted network. We call $(T,{\o})^{B,H}$ the finite random network
obtained by the same sorting procedure. Namely, $(T,{\o})^{B,H}$ consists of
the subtree with vertices in $J_{B,H}$, with the marks inherited from the
infinite tree. If an edge is not present in $(T,{\o})^{B,H}$, we assign to it
the mark $+\infty$.

We say that the sequence of random finite networks $(G_n,1)^{B,H}$ converges
in distribution (as $n\to\infty$) to the random finite network
$(T,{\o})^{B,H}$ if the joint distributions of the marks converge weakly. To
make this precise we have to add the points $\{\pm\infty\}$ as possible values
for each mark, and continuous functions on the space of marks have to be
understood as functions such that the limit as any one of the marks diverges
to $+\infty$ exists and coincides with the limit as the same mark diverges to
$-\infty$. We may define $\overline \dC = \dC \cup \{\pm \infty\} $. The next
proposition generalizes \cite[Section 3]{aldous92}, for a proof see
\cite[Proposition 2.6]{bordenave-caputo-chafai-ii} (the proof there is stated
for a measure $\rho$ on $\dR$, the complex case extends verbatim).
\begin{prop}[Local weak convergence to a tree] \label{prop:LWC} %
  Let $(\xi^n_{i,j})_{1 \leq i \leq j \leq n}$ be a collection of i.i.d.\ 
  random variables in $\overline \dC$ and set $\xi^n _ {j,i} = \xi^n_{i,j}$.
  Let $\r$ be a Radon measure on $\dC$ with no mass at $0$ and assume that
  \begin{equation}\label{asso}
    n \dP ( \xi^n_{12} \in \cdot ) 
    \underset{n\to\infty}{\weak}
    \r .%
  \end{equation}
  Let $G_n$ be the complete network on $\{1,\ldots,n\}$ whose mark on edge
  $\{i,j\}$ equals $\xi^n_{i j}$, and $T$ a realization of $\pwit(\r)$. Then,
  for all integers $B,H$,
  \[
  (G_n,1) ^{B,H} \underset{n\to\infty}{\weak}  (T,{\o}) ^{B,H}.
  \] 
\end{prop}

Now, we shall extend the above statement to directed networks. More precisely, let $(\xi^n_{i , j})_{1 \leq i , j \leq n}$ be i.i.d.\ real random
variables. We consider the complete graph $\bar G_n$ on $V_n$ whose weight on
edge $\{i,j\}$ equals, if $i \leq j$, $(\xi^n_{i,j}, \xi^n_{j,i}) \in \dR^2$.
As above, we partially index the vertices of $(\bar G_n,1)$ as elements in
\[
J_{B,H} = \cup_{\ell=0}^H \{1,\cdots, B\}^\ell \subset \dN^f,
\]
the indexing being given by an injective map $\sigma_n$ from $J_{B,H}$ to
$V_n$ such that $ \sigma_n^{-1} (1) = {\o}$. The difference with the above
construction, is that the vertex $v\in V_n \backslash \{1\}$ is given the
index $(k) = \sigma_n^{-1} (v)$, $1 \leq k\leq B$, if $\min ( | \xi^n_{1,v}| ,
| \xi^n_{v,1}| )$ has the $k^\text{th}$ smallest value among $\{\min ( |
\xi^n_{1,j}| , | \xi^n_{j,1}| ) , \,j\neq 1\}$.

Similarly, let $(T,{\o})$ be the infinite random rooted network with
distribution $\pwit(\rho)$. This time we do not remove the Bernoulli marks
$(\e_v)_{v \in \dN^f}$ and define the weight on edge $\{v , vk \}$ as $(
y_{vk} , \infty)$ if $\e_{vk} = 1$ and $( \infty , y_{vk})$ if $\e_{vk} = 0$.
Again, we call $(T,{\o})^{B,H}$ the finite random network obtained by the
sorting procedure : $(T,{\o})^{B,H}$ consists of the subtree with vertices in
$J_{B,H}$, with the marks inherited from the infinite tree.

We apply proposition \ref{prop:LWC} to the complete network $G_n$ with mark on edge $\{i,j\}$ equals, if $i \leq j$, to $\min (|�\xi^n_{i,j}| , |�\xi^n_{j,i}|�)$. This network satisfies \eqref{asso} with $2\rho$. We remark that if $u,v \in J_{B,H}$ then from \eqref{asso}, $\max( |\xi^n _{\sigma_n (u) ,\sigma_n (v)}|,  |\xi^n _{\sigma_n (v) ,\sigma_n (u)}| ) $ diverges weakly to infinity. We also notice that, given $(G_n,1) ^{B,H}$, with equal probability  $|\xi^n _{\sigma_n (u) ,  \sigma_n (v)}|$ is larger or less than $|\xi^n _{\sigma_n (v) ,  \sigma_n (u)}| $. We deduce the following. 

\begin{cor}[Local weak convergence to a tree] \label{cor:LWC} %
  Let $\r$ be a Radon measure on $\dC$ with no mass at $0$. Let
  $(\xi^n_{i,j})_{1 \leq i , j \leq n}$ be a collection of i.i.d.\  random
  variables in $\overline \dC$ such that (\ref{asso}) holds. Let $\bar G_n$ be
  the complete network on $\{1,\ldots,n\}$ whose mark on edge $\{i, j\}$
  equals, if $i \leq j$, $(\xi^n_{i,j}, \xi^n_{j,i})$, and $T$ a realization
  of $\pwit(2\r)$. Then, for all integers $B,H$, 
  \[ 
  (\bar G_n,1) ^{B,H}
  \underset{n\to\infty}{\weak} (T,{\o}) ^{B,H}.
  \]
\end{cor}

We may now prove theorem \ref{th:locconvPWIT}. 
\begin{proof}[Proof of theorem \ref{th:locconvPWIT}]
  We argue as in the proof of theorem 2.3(i) in \cite[Section
  2]{bordenave-caputo-chafai-ii}. We first define the weights
  $(\xi^n_{i,j})_{i,j \in \dN^f}$ as follows. For integers $1 \leq i,j \leq
  n$, we set
  \[
  \xi^n_{i,j} = A_{i,j}^{-\alpha} = a_n^\alpha X_{i,j}^{-\alpha},
  \] 
  with the convention that $\xi^n_{i,j} = \infty$ if $X_{i,j} = 0$. For this
  choice, by assumption (H1), \eqref{asso} holds with $\rho = \ell_{\theta}$
  and $\ell_{\theta}$ in \eqref{eq:elltheta}. If $i$ or $j$ is in $\dN^f
  \backslash \{ 1, \cdots, n\}$, we set $\xi^n_{i,j} = \infty$.

  Let $\bar G_n$ denote the complete network on $\{1, \cdots, n\}$ with marks
  $(\xi^n_{i,j}, \xi^n_{j,i})$ on edge $\{ i , j\}$, if $i \leq j$. From
  Corollary \ref{cor:LWC}, for all $B,H$, $(\bar G_n,1)^{B,H}$ converges
  weakly to $(T,\o)^{B,H}$, where $T$ has distribution
  $\pwit(2\ell_{\theta})$. Let $A$ be the random operator associated to $T$.
  
  Let $\si_n^{B,H}$ be the map $\si_n$ associated to the network $(\bar
  G_n,1)^{B,H}$. The maps $\si_n$ are arbitrarily extended to a bijection
  $\dN^f\to\dN^f$. From the Skorokhod Representation Theorem we may assume
  that $(\bar G_n,1)^{B,H}$ converges a.s.\ to $(T,\o)^{B,H}$ for all $B,H$.
  Thus we may find sequences $B_n,H_n$ tending to infinity and a sequence of
  bijections $\wt \si_n := \si_n^{B_n,H_n}$ such that $(B_n ^{H_n +1} - 1 )/
  (B_n - 1) \leq n$ and such that for any pair $u,v\in\bbN^f$ we have
  $\xi^n_{\wt\sigma_n (u), \wt\sigma_n (v)}$ which converge a.s. to
  \[
  \left\{ \begin{array}{ll}
      y_{uk}  & \hbox{ if for some integer $k$, }  v = u k \hbox{ and } \e_{uk} = 1 \\
      y_{vk}  & \hbox{ if for some integer $k$, }  u = v k \hbox{ and } \e_{vk} = 0 \\
      \infty & \hbox{ otherwise }
    \end{array} \right. 
  \]
  It follows that a.s.\
  \begin{equation*} \langle \delta_u , \wt\sigma_n^{-1} A_n \wt\sigma_n
    \delta_v \rangle = { \xi^n_{\wt\sigma_n (u),\wt \sigma_n (v)}
    }^{-1/\alpha} \to \langle \delta_u , A \delta_v \rangle\,.
    \label{pointw}
  \end{equation*}
  For any $v$, set $\psi_n^v := \wt\si_n^{-1} A_n \wt\si_n \d_v$. To prove
  theorem \ref{th:locconvPWIT}, it is sufficient to show that for any
  $v\in\dN^f$, $\psi_n^v\to A \d_v$ in $\ell^2( \dN^f)$ almost surely as $n$
  goes to infinity, i.e.,
  \begin{equation*}\label{l2con} 
    \sum_{u}
    (\left\langle\d_u,\psi_n^v\right\rangle -
    \left\langle\d_u,A \d_v\right\rangle)^2 \to 0\,.
  \end{equation*}
  From what precedes, we know that $\left\langle \d_u,\psi_n^v\right\rangle\to
  \left\langle \d_u,A \d_v\right\rangle$ for every $u$. The claim follows if
  we have (almost surely) uniform (in $n$) square-integrability of
  $(\left\langle\d_u,\psi_n^v\right\rangle)_u $. This in turn follows from
  Lemma 2.4(i) and Lemma 2.7 in \cite{bordenave-caputo-chafai-ii}. \end{proof}

\subsection{Convergence of the resolvent matrix}

Let $A_n$ and $A$ be as in theorem \ref{th:locconvPWIT}. From proposition
\ref{prop:sa}, we may almost surely define the resolvent $R$ of the bipartized
random operator of $A$. For $U = U (z,\eta) \in \dH_+$, we set
\begin{equation}\label{eq:defabc}
  R(U)_{\o \o} = \Pi_{\o} R(U)  \Pi_{\o}^* = %
  \begin{pmatrix} a (z,\eta) & b(z,\eta)  \\  
    b' (z,\eta) &  c (z,\eta) 
  \end{pmatrix}. 
\end{equation}
We define similarly, $R_n(U) = (B_n (z) - \eta )^{-1}$, the resolvent of
$B_n$, the bipartized operator of $A_n$. We set $R_n(U)_{11} = \Pi_{1} R_n(U)
\Pi_{1}^*$.

\begin{thm}[Convergence of the Resolvent matrix] \label{th:convRes} Let $A_n$
  and $A$ be as in theorem \ref{th:locconvPWIT}. For all $U = U ( z, \eta) \in
  \dH_+$,
  \[
  R_{n} (U)_{11} \underset{n\to\infty}{\weak}R(U)_{\o \o}.
  \] 
\end{thm}

\begin{proof}[Proof of theorem \ref{th:convRes}]
  We apply proposition \ref{prop:sa}, theorem \ref{th:locconvPWIT} and the``in
  distribution" version of theorem \ref{th:strongres}.
\end{proof}

\subsection{Proof of theorem \ref{th:mpz}}\label{ss:mpz}

Again, we consider the sequence of random $n \times n$ matrices $(A_n)$
defined in introduction by \eqref{eq:defAn}.

\begin{thm}\label{th:convSV}
  For all $z \in \dC_+$, almost surely the measure $\cnu_{A_n - z} ( dx)$
  converges weakly to a measure $\cnu_{\alpha,z} ( dx)$ whose Cauchy-Stieltjes
  transform is given, for $\eta \in \dC_+$, by
  \[
  m_{ \cnu_{\alpha,z }} ( \eta )  = \dE a (z, \eta), 
  \]
  where $ a (z,\eta) $ was defined in \eqref{eq:defabc}.
\end{thm}

\begin{proof}
  For every $z \in \dC$, by proposition \ref{prop:sa}, the operator $B(z)$ is
  a.s. self-adjoint. It implies that there exists a.s. a measure on $\dR$,
  $\nu_{\o,z}$, called the spectral measure with vector $\d_{\o}$, such that
  for all $\eta \in \dC_+$,
  \[
  a( z, \eta) = \langle\d_{\o} , R ( U ) \d_{\o} \rangle %
  = \int \frac{\nu_{\o,z}(dx)}{x - \eta} = m_{ \nu_{\o,z}} (\eta).
  \]
  We define $R_n$ as the resolvent matrix of $B_n$, the bipartized matrix of
  $A_n$. For $U = U(z,\eta)\in \dH_+$, we write $R_{n} ( U)_{kk}
  = \begin{pmatrix} a_k & b_k \\ b'_k & c_k \end{pmatrix}$. By theorem
  \ref{th:muabc},
  \[
  m_{\dE \cnu_{A_n -z}} (\eta) = \dE a_{1} (z,\eta). 
  \]
  By lemma \ref{le:Rbounded}, for $U \in \dH_+$, the entries of the matrix
  $R_n( U)_{11}$ are bounded. It follows from theorem \ref{th:convRes} that for
  all $U \in \dH_+$,
  \[
  \lim_{n\to\infty} \dE R_n ( U) _{11} = \dE \begin{pmatrix} a & b \\ b' &
    c \end{pmatrix},
  \]
  where the limit matrix was defined in \eqref{eq:defabc}. Hence, for all $z
  \in \dC_+$,
  \[
  \lim_{n\to\infty}  m_{\dE \cnu_{A_n -z}} (\eta) = \dE a (z,\eta). 
  \]
  We deduce that $\dE \cnu_{A_n-z}$ converges to the measure $\nu_{\alpha,z} =
  \dE \nu_{\o,z}$. This convergence can be improved to almost sure by showing
  that the random measure $\cnu_{A_n-z}$ concentrates around its mean. This is
  done by applying Borel-Cantelli Lemma and lemma \ref{le:concspec} to the
  matrix $B_n(z)$ whose spectral measure equals $\cnu_{A_n-z}$, see
  \eqref{eq:mubnz}.
\end{proof}

Theorem \ref{th:mpz} is a corollary of the above theorem up to the fact that
$\dE a(z,\eta)$ does not depend on the measure $\theta$ which appears in (H1).
The latter will be a consequence of the forthcoming theorem \ref{th:BC}.

\section{Convergence of the spectral measure}\label{se:girko}

\subsection{Tightness}\label{ss:tightness}

In this paragraph, we prove that the counting probability measures of the
eigenvalues and singular values of the random matrices $(A_n)$ defined by
\eqref{eq:defAn} are a.s.\ tight. For ease of notation, we will often write $A$ in place of $A_n$. 

\begin{lem}[Tightness]\label{le:largeSV}
  If (H1) holds, there exists $r>0$ such that for all $z \in \dC$, a.s.\
  \[
  \varlimsup_{n\to\infty}
  \int_0^\infty\!t^r\,\nu_{A-zI}(dt) < \infty,
  \quad\text{and thus}\quad (\nu_{A-zI})_{n\geq1} \text{ is tight}.
  \]
  Moreover, a.s.\
  \[
  \varlimsup_{n\to\infty}
  \int_{\mathbb{C}}\!|z|^r\,\mu_{A}(dz) < \infty,
  \quad\text{and thus}\quad (\mu_{A})_{n\geq1} \text{ is tight}.
  \]
\end{lem}

\begin{proof}  
  In both cases, the a.s.\ tightness follows from the moment bound and the
  Markov inequality. The moment bound on $\mu_A$ follows from the statement on
  $\nu_{A}$ (take $z=0$) by using the Weyl inequality \eqref{eq:weyl3}. It is
  therefore enough to establish the moment bound on $\nu_{A-zI}$ for every
  $\mathbb{C}$. Let us fix $z\in\dC$ and $r>0$. By definition of $\nu_{A-zI}$
  we have
  \[
  \int_0^\infty\!t^r\,\nu_{A-zI}(dt)
  =\frac{1}{n}\sum_{k=1}^n s_k(A-zI)^r.
  \]
  From \eqref{eq:basic1} we have $s_k(A-zI)\leq s_k(A)+|z|$ for every $1\leq
  k\leq n$, and one can then safely assume that $z=0$ for the proof. By using
  \eqref{eq:zhan} we get for any $0\leq r\leq 2$,
  \[
  \int_0^\infty\!t^r\,\nu_A(dt)
  \leq 
  Z_n:=\frac{1}{n}\sum_{i=1}^n Y_{n,i}
  \quad\text{where}\quad
  Y_{n,i}:=\PAR{\sum_{j=1}^na_n^{-2}|X_{ij}|^2}^{r/2}.
  \]
  We need to show that $(Z_n)_{n\geq1}$ is a.s.\ bounded. Assume for the
  moment that 
  \begin{equation}\label{eq:r}
    \sup_{n\geq1}\dE(Y_{n,1}^4)<\infty
  \end{equation}
  for some choice of $r$. Since $Y_{n,1},\ldots,Y_{n,n}$ are i.i.d.\ for every
  $n\geq1$, we get from \eqref{eq:r} that
  \[
  \dE( (Z_n-\dE Z_n)^4) %
  = n^{-4}\dE \PAR{\sum_{1 \leq i,j \leq n}(Y_{n,i}-\dE Y_{n,i})^2(Y_{n,j}-\dE
    Y_{n,j})^2} %
  =O(n^{-2}).
  \]
  Therefore, by the monotone convergence theorem, we get
  $\dE (\sum_{n\geq1}(Z_n-\dE Z_n)^4)<\infty$, which gives
  $\sum_{n\geq1} (Z_n-\dE Z_n)^4<\infty$ a.s.\ and thus
  $Z_n-\dE Z_n\to0$ a.s.\ Now the sequence
  $(\dE Z_n)_{n\geq1}=(\dE Y_{n,1})_{n\geq1}$ is bounded by
  \eqref{eq:r} and it follows that $(Z_n)_{n\geq1}$ is a.s.\ bounded.

  It remains to show that \eqref{eq:r} holds, say if $0<4r<\alpha$. To this
  end, let us define
  \[
  S_{n,a,b}:= \sum_{j=1}^n a_n^{-2}|X_{1j}|^2
  \mathds{1}_{\{a_n^{-2}|X_{1j}|^2 \in [a,b)\}}
  \quad\text{for every $a<b$.}
  \]
  Now $Y_{n,1}^4=(S_{n,0,\infty})^{2r}=(S_{n,0,1}+S_{n,1,\infty})^{2r}$ and thus,
  \begin{equation}\label{eq:Yn}
    \dE(Y_{n,1}^4) \leq  %
    2^{2r-1} \BRA{\dE (S_{n,0,1}^{2r} )     + \dE (S_{n,1,\infty}^{2r})}.
  \end{equation}
  We have $\sup_{n\geq1} \dE (S_{n,0,1}^{2r}) < \infty$. Indeed, since $2 r < 1$, by
  the Jensen inequality, 
  \[
  \dE (S_{n,0,1}^{2r}) \leq (\dE S_{n,0,1} ) ^{2r}
  \]
  and by lemma \ref{le:XV2}, 
  \[
  \dE S_{n,0,1} \sim_n \alpha / (2- \alpha).
  \]
  For the second term of the right hand side of \eqref{eq:Yn}, we set
  \[
  M_n := \max_{1 \leq j \leq n} a_n^{-1}|X_{1j}| \ind_{\{a_n^{-1}|X_{1j}|
    >1\}} %
  \; \text{and } \; %
  N_n := \#\{1\leq j\leq n\text{ s.t. } a_n^{-1}|X_{1j}| >1\} .
  \] 
  From H\"older inequality, if $1/p + 1/q =1$, we have
  \begin{equation}\label{eq:holderS}
    \dE (S_{n,1,\infty}^{2r}) %
    \leq \dE \left( N_n^{2r} M_n^{4r} \right) %
    \leq  \left( \dE  N_n^{2rp} \right)^{1/p} \left(\dE  M_n^{4rq} \right)^{1/q}.
  \end{equation}
  Recall that $\dP ( |X_{12}|> a_n ) = (1 + o(1) )/ n \leq 2 / n$ for $n \gg 1$.
  By the union bound, for $ n \gg 1$,
  \begin{eqnarray*}
    \dP ( N_n \geq k )  \leq   \binom{n}{k}  \dP ( |X_{12}|> a_n ) ^k  
    \leq  \frac {n^k}{k!} \frac{ 2^k}{n^k} = \frac { 2^k }{k ! }. 
  \end{eqnarray*}
  In particular, we have $\sup_{n\geq1} \dE N_n ^\eta < \infty$ for any $\eta >0$.
  Similarly, since the function $L$ is slowly varying, for $n \gg 1$ and all
  $t\geq 1$, we have
  \[
  \dP ( M_n \geq t ) \leq n \dP ( |X_{12}|> t a_n ) %
  = n a_n^{-\alpha} t^{-\alpha} L( a_n t ) \leq 2 t^{-\alpha}.
  \]
  It follows that if $\gamma < \alpha$, $\sup_{n\geq1} \dE M_n^{\gamma} < \infty$.
  Taking $p$ and $q$ so that $4rq < \alpha$, we thus conclude from
  \eqref{eq:holderS} that $\sup_{n\geq1} \dE (S_{n,1,\infty}^{2r}) < \infty$.
\end{proof}

\subsection{Invertibility}\label{th:invertibility}

In this paragraph, we find a lower bound for the smallest singular value of
the random matrix $A-zI$ where $A$ is defined by \eqref{eq:defAn}, in other
words an upper bound on the operator norm of the resolvent of $A$. Such lower
bounds on the smallest singular value of random matrices were developed in
the recent years by using Littlewood-Offord type problems, as in
\cite{MR2409368,MR2684367} and \cite{MR2407948}. The available results require
moments assumptions which are not satisfied when the entries have heavy tails.
Here we circumvent the problem by requiring the bounded density hypothesis
(H3). The removal of this hypothesis can be done by adapting
the Rudelson and Vershynin approach already used by G\"otze and Tikhomirov
\cite{gotze-tikhomirov-new}.

\begin{lem}[Invertibility]\label{le:smallestSV}
  If (H3) holds then for some $r > 0$, every $z\in\dC$, a.s.\
  \[
  \varliminf_{n\to\infty} n^rs_n(A-zI)= +\infty.
  \]
\end{lem}

\begin{proof} 
  For every $x,y\in\dC^n$ and $S\subset\dC^n$, we set $x\cdot
  y:=x_1\overline{y_1}+\cdots+x_n\overline{y_n}$ and $\NRM{x}_2:=\sqrt{x\cdot
    x}$ and $\mathrm{dist}(x,S):=\min_{y\in S}\NRM{x-y}_2$. Let
  $R_1,\ldots,R_n$ be the rows of $A-zI$ and set
  \[
  R_{-i}:=\mathrm{span}\{R_j;j\neq i\}
  \]
  for every $1\leq i\leq n$. From lemma \ref{le:rvdist} we have
  \[
  \min_{1\leq i\leq n} \dist  (R_i,R_{-i}) \leq \sqrt{n}\,s_n(A-zI)
  \]
  and consequently, by the union bound, for any $u\geq0$,
  \[
  \dP(\sqrt{n}\,s_n(A-zI)\leq u) %
  \leq n\max_{1\leq i\leq n}\dP( \dist (R_i,R_{-i})%
  \leq u).
  \]
  Let us fix $1\leq i\leq n$. Let $Y_i$ be a unit vector orthogonal to $R_{-i}$.
  Such a vector is not unique. We just pick one. This defines a random
  variable on the unit sphere $\mathbb{S}^{n-1}=\{x\in\dC^n:\NRM{x}_2=1\}$. By
  the Cauchy--Schwarz inequality,
  \[
  |R_i\cdot Y_i|\leq \NRM{\pi_i(R_i)}_2\NRM{Y_i}_2= \dist (R_i,R_{-i})
  \]
  where $\pi_i(\cdot)$ is the orthogonal projection on the orthogonal complement of
  $R_{-i}$. Let $\nu_i$ be the distribution of $Y_i$ on $\mathbb{S}^{n-1}$. Since
  $Y_i$ and $R_i$ are independent, for any $u\geq0$,
  \[
  \dP( \dist (R_i,R_{-i})\leq u) %
  \leq \dP(|R_i\cdot Y_i|\leq u) %
  = \int_{\mathbb{S}^{n-1}}\!\!\!\dP(|R_i\cdot y|\leq u)\,d\nu_i(y)
  \]
  Let us first consider the case where $X_{11}$ has a bounded density
  $\varphi$ on $\dC$. Since $\NRM{y}_2=1$ there exists an index
  $j_0\in\{1,\ldots,n\}$ such that $y_{j_0}\neq 0$ with
  $\ABS{y_{j_0}}^{-1}\leq \sqrt{n}$. The complex random variable $R_i\cdot y$
  is a sum of independent complex random variables and one of them is
  $a_n^{-1}X_{ij_0}\,\overline{y_{j_0}}$, which is absolutely continuous with
  a density bounded above by $a_n\sqrt{n}\,\NRM{\varphi}_\infty$.
  Consequently, by a basic property of convolutions of probability measures,
  the complex random variable $R_i\cdot y$ is also absolutely continuous with
  a density $\varphi_i$ bounded above by $a_n\sqrt{n}\,\NRM{\varphi}_\infty$,
  and thus
  \[
  \dP(|R_i\cdot y|\leq u) %
  = \int_{\dC}\,\ind_{|s| \leq u }\varphi_i(s)\,ds %
  \leq \pi u^2\,a_n\sqrt{n}\,\NRM{\varphi}_\infty.
  \]  
  Therefore, for every $b>0$,
  \[
  \dP(s_n(A-zI)\leq n^{-b-1/2}) = O(n^{3/2-2b}a_n)
  \]
  where the $O$ does not depend on $z$. By taking $b$ large enough, the first
  Borel-Cantelli lemma implies that there exists $r>0$ such that a.s.\ for
  every $z\in\dC$ and $n\gg1$,
  \[
  s_n(A-zI)\geq n^{-r}.
  \]
  It remains to consider the case where $X_{11}$ has a bounded density
  $\varphi$ on $\dR$. As for the complex case, let us fix
  $y\in\mathbb{S}^{n-1}$. Since $\NRM{y}_2=1$ there exists an index
  $j_0\in\{1,\ldots,n\}$ such that $\ABS{y_{j_0}}^{-1}\leq \sqrt{n}$. Also,
  either $|\mathfrak{Re}(y_{j_0})|^{-1}\leq\sqrt{2n}$ or
  $|\mathfrak{Im}(y_{j_0})|^{-1}\leq\sqrt{2n}$. Assume for instance that
  $|\mathfrak{Re}(y_{j_0})|^{-1}\leq\sqrt{2n}$. We observe that for every
  $u\geq0$,
  \[
  \dP(|R_i\cdot y|\leq u) \leq \dP(|\mathfrak{Re}(R_i\cdot y)|\leq u).
  \]
  The real random variable $\mathfrak{Re}(R_i\cdot y)$ is a sum of independent
  real random variables and one of them is
  $a_n^{-1}X_{ij_0}\mathfrak{Re}(y_{j_0})$, which is absolutely continuous
  with a density bounded above by $a_n\sqrt{2n}\,\NRM{\varphi}_\infty$.
  Consequently, by a basic property of convolutions of probability measures,
  the real random variable $\mathfrak{Re}(R_i\cdot y)$ is also absolutely
  continuous with a density $\varphi_i$ bounded above by
  $a_n\sqrt{2n}\,\NRM{\varphi}_\infty$. Therefore, we have for every $u\geq0$,
  \[
  \dP(|\mathfrak{Re}(R_i\cdot y)|\leq u) %
  = \int_{[-u,u]}\,\varphi_i(s)\,ds %
  \leq 2^{3/2}a_n\sqrt{n}\,u\,\NRM{\varphi}_\infty.
  \]
  We skip the rest of the proof, which is identical to the complex case. 
\end{proof}

\subsection{Distance from a row to a vector space}

In this paragraph, we give two lower bounds on the distance of a row of the
random matrix $A - z$ defined by \eqref{eq:defAn} to a vector space of not too
large dimension. The first ingredient is an adaptation of Proposition 5.1 in
Tao and Vu \cite{MR2722794}.

\begin{prop}[Distance of a row to a subspace]\label{prop:distXW}
  Assume that (H1) holds. Let $0 < \gamma < 1/2$, and let $R$ be a row of $a_n
  ( A - z)$. There exists $\delta >0$ depending on $\a,\gamma$ such that for
  all $d$-dimensional subspaces $W$ of $\dC^n$ with $ n - d \geq n ^{1 -
    \gamma}$, one has
  \[
  \dP \left( \dist ( R, W) \leq n ^{(1-2\g)/\a} \right) \leq e^{ - n^{\delta}}.
  \]
\end{prop}

The proof of proposition \ref{prop:distXW} is based on a concentration
estimate for the truncated variables $X_{1i} \ind_{\{| X_{1i} | \leq b_n\}}$
for suitable sequences $b_n$. We first recall a concentration inequality of
Talagrand.

\begin{thm}[Talagrand concentration inequality {\cite{MR1361756} and
    \cite[Corollary 4.10]{ledoux01}}]\label{th:talagrand}
  Let us denote by $\dD:=\{z\in\mathbb{C};|z|\leq 1\}$ the complex unit disc
  and let $P$ be a product probability measure on the product space $\dD^n$.
  Let $F : \dD^n \to \dR$ be a Lipschitz convex function on $\dD^n$ with
  $\LIP{F}\leq 1$. If $M(F)$ is a median of $F$ under $P$ then for every
  $r\geq0$,
  \[
  P \left(   | F - M(F) | \geq r   \right) \leq 4 e^{  - r ^2 / 4}.
  \]
\end{thm}

\begin{proof}[Proof of proposition \ref{prop:distXW}]
  We first perform some pre-processing of the vector $R$ as in Tao-Vu
  \cite{MR2722794}. To fix ideas, we may assume that $R$ is the first row of
  $a_n ( A -z) $. Then $R = X_1 - z a_n e_1$ where $X_1$ is the first row of
  $X = a_n A$. We then have
  \[
  \dist ( R , W) %
  \geq \dist ( X_1 - z a_n e_1 , \vect ( W , e_1) ) %
  = \dist ( X_1 , W_1).
  \]
  where we have set $W_1 = \vect ( W, e_1)$. Note that $d \leq \dim W_1 \leq d
  +1$.

  For any sequence $b_n$, from the Markov inequality,
  \begin{eqnarray}
    \dP \left( \sum_{i =1}^n\ind_{ \{ |X_{1i} | \geq  b_n  \}}  \geq \sqrt n  \right) & \leq  & e^{-  \sqrt n } \left(\dE  e^{  \ind_{  |X_{11} | \geq b_n  }} \right)^n \nonumber\\
    & \leq & e^{- \sqrt n  } \left( 1 +   e L(b_n) b_n^{-  \alpha  } \right)^n \nonumber\\
    & \leq & e^{  -   \sqrt n   + e n L(b_n) b_n  ^{-\alpha }  }.
    \label{chebbo}
  \end{eqnarray}
  Choose $b_n = a_n n^{-2\g / \alpha }$. Clearly,
  $b_n/n^{(1-2\g)/\a}\in[n^{-\e},n^\e]$ eventually for all $\e>0$.
 
  Let $\cJ$ denote the set of indexes $i$ such that $|X_{1i}|\leq b_n$.
  From (\ref{chebbo}) we see that, for some $\d>0$:
  \begin{equation*}
    \dP(|\cJ|<n-\sqrt n) \leq e^{-n^{\d}}\,.
  \end{equation*}
  It follows that it is sufficient to prove the statement conditioned on the
  event $\{ |\cJ|\geq n-\sqrt n\}$. In particular, we shall prove that for any
  fixed $I\subset \{1,\dots,n\}$, such that $|I|\geq n-\sqrt n$,
  \begin{equation}\label{fincla}
    \dP \left( \dist ( X_1, W_1) \leq  n ^{(1-2\g)/\a} \tc \cJ=I \right) %
    \leq e^{ - n^{\delta}}. 
  \end{equation}
  Without loss of generality, we assume that $I = \{ 1, \cdots, n' \}$ with $n
  ' \geq n - \sqrt n $. Let $\pi_I $ be the orthogonal projection on $\vect (
  e_ i : i \in I )$. If $W_2 = \pi_I ( W_1)$, we find $d - \sqrt n \leq \dim (
  W_2) \leq \dim (W_1) \leq d+1$ and
  \[
  \dist ( X_1 , W_1) \geq \dist( \pi_I (X_1) , W_2).
  \]
  Note that $\pi_I(X_1)$ is simply the vector $X_{1i}$, $i=1,\dots,n'$. We set
  \[
  W' = \hbox{span} ( W_2, \dE [ \pi_I (X_1) \tc \cJ=I ])\,,\;\;\; Y = \pi_I
  (X_1) - \dE[ \pi_I (X_1)\tc\cJ=I]\,,
  \]
  so that $d -\sqrt n \leq \dim ( W') \leq d +2$ and
  \[
  \dist( \pi_I (X_1) , W_2) \geq  \dist( Y , W').  
  \]
  Let $P$ denote the orthogonal projection matrix to the orthogonal complement
  of $W'$ in $\dC^{n'}$. We have $\dist^2 ( Y , W') = \sum_{i,j} Y_i P_{ij}
  \bar Y_j $, and, since $Y=(Y_i)_{1\leq i\leq n'}$ is a mean zero vector
  under $\dP(\cdot\tc\cI=I)$,
  \begin{align*}
    \dE [\dist^2 ( Y , W')\tc\cJ=I] & = \dE \Big[\sum_{i,j} Y_i P_{ij} \bar Y_j \tc \cJ=I\Big]\\ &= 
    \sum_{i=1}^{n'} P_{ii} \dE[| Y_i| ^2\tc\cJ=I]  =   \dE[| Y_1| ^2\tc\cJ=I]  \,\tr P.  
  \end{align*}
  We have for any $\e>0$ and for $n \gg 1$:
  \begin{align*}
    \dE[| Y_1| ^2\tc \cJ=I ] = \dE[| X_{11}| ^2 \tc \cJ=I ] - \left( \dE[|
      X_{11}| \tc \cJ=I ]\right) ^2 \geq b_n ^{ 2 - \alpha }\,n^{-\e}\,,
  \end{align*}
  where the last bound follows from lemma \ref{le:XV2}, since by independence
  one has 
  \[
  \dE[| X_{11}| ^2\tc \cJ=I ] =\dE[| X_{11}| ^2\tc |X_{11}|\leq b_n ]\,,
  \] 
  and $|\dE[ X_{11}\tc \cJ=I ]|^2 = |\dE[ X_{11}\tc |X_{11}|\leq b_n
  ]|^2$ is $O(1)$ if $\a>1$, while (by lemma \ref{le:XV2}) it is
  $O(b_n^{2-2\a+\e})$ for any $\e>0$, if $\a\in(0,1]$.

  Using $ \tr P = n'- \dim ( W') \geq \frac12\,(n-d)$, it follows that, for
  any $\e>0$, for $n \gg 1 $:
  \begin{equation}\label{eq:dist2YW}
    \dE [\dist^2 ( Y, W')\tc  \cJ=I] \geq c L(b_n)   b_n ^{ 2 -  \alpha } (n - d)\geq n^{q(\e)}\,,
  \end{equation}
  where $q:=(1-2\g)\frac2\a +\g -\e$.

  Under $\dP(\cdot\tc\cJ=I)$, the vector $(Y_1 /b_n , \cdots, Y_{n'} /b_n )$
  is a vector of independent variables on $\dD^{n'}$, where $\dD$ be the unit
  complex ball. We consider the 
  function $ F : x \mapsto \dist ( x , W')$. The mapping $F$ is $1$-Lipschitz
  and convex. From theorem \ref{th:talagrand}, we deduce that
  \begin{equation}\label{eq:talagranddist}
    \dP (\left|\dist(Y,W')-M(\dist(Y,W'))\right|\geq r\tc\cJ=I)%
    \leq 4 e^{ - \frac{r^2}{8 b^2_n} } 
  \end{equation}
  where $M(\dist ( Y , W')) $ is a median of $\dist ( Y , W')$ under
  $\dP(\cdot\tc\cJ=I)$.

  It follows that, for e.g.\ $\delta =\g/2$, taking $\e=\g/4$ in
  \eqref{eq:dist2YW}, we obtain $q(\e) = (1-2\g)\frac2\a + \d + \e$, and
  therefore there exists $c >0$ such that $n \gg 1$,
  \begin{equation}\label{eq:cdel}
    b_n^{-2}\,\dE [\dist^2 ( Y, W')\tc  \cJ=I] \geq c\,\frac{n^{q(\e)}}{b_n^2} \geq c\, n ^{\delta}\,.
  \end{equation}
  From \eqref{eq:talagranddist} it follows that
  \[
  \dE\left[ \left| M(\dist(Y,W'))-\dist(Y,W')\right|^2\tc\cJ=I\right] %
  = O \left(\,b^2_n \right)
  \]
  From the Cauchy-Schwarz inequality we then have
  \begin{align*}
    &   \left|M (  \dist ( Y , W'))  -  \sqrt {\dE [\dist^2 ( Y , W')\tc\cJ=I] } \,\right|^{\,2} \\
    &\qquad \leq \dE\left[ \left| M  (  \dist ( Y , W'))  -   \dist( Y , W')  \right|^2 \tc \cJ=I\right] = O \left( \,b^2_n \right)  \,.
  \end{align*}
  The above estimates, with \eqref{eq:dist2YW} and \eqref{eq:cdel}, imply that
  $M ( \dist ( Y , W')) \geq \frac12\,n^{q(\e)/2}$ for $n \gg 1$. Therefore,
  for $n \gg 1$,
  \begin{align*}
    &\dP \left( \dist ( Y , W') \leq n^{(1-2\g)/\a}
      \tc\cJ=I\right)\\
    &\qquad \;\leq \dP \left( \left| M ( \dist ( Y , W')) - \dist( Y , W')
      \right| \geq \frac14\,n^{q(\e)/2} \tc\cJ=I\right)\,.
  \end{align*}
  The desired conclusion (\ref{fincla}) now follows from
  \eqref{eq:talagranddist} and \eqref{eq:cdel}.
\end{proof}

So far we have shown that under assumption (H1), the distance of a row to a
space with codimension $n-d\geq n^{1-\g}$ is at least $n^{(1-2\g)/\a}$ with
large probability. We want a sharper estimate, namely at the order $
n^{1/\a}$. We will obtain such a bound in a weak sense in the forthcoming
proposition \ref{prop:distXW2}. Furthermore, we shall require assumption (H2)
to do so. We start with some preliminary facts.

Below we write $Z=Z^{(\b)}$, $\b\in(0,1)$, for the one-sided $\b$-stable
distribution such that for all $s \geq 0$,
\[
\dE \exp ( - s Z_i) = \exp ( - s^{\b} )\,.
\] 
From the standard inversion formula, for $m>0$
\[
y^{-m} = \G(m)^{-1}\int_0^\infty x^{m-1}\,\nep{-x\,y}\,dx\,,
\]
we see that all moments 
\begin{equation}\label{um}
\dE[Z^{-m}] =\G(m)^{-1}\int_0^\infty x^{m-1}\,\nep{-x^\b}\,dx
\end{equation} 
are finite for $m>0$. Also, recall that if $(Z_i)_{1 \leq i \leq n}$ is an
i.i.d.\  vector with distribution $Z$ then, for every $(w_i)_{1 \leq i \leq n}
\in \dR_+^n$, in distribution
\begin{equation}\label{eqdis}
  \sum_{i = 1} ^n w_i Z_{i } %
  \stackrel{d}{=} \left( \sum_{i=1} ^n w_i^{\b} \right)^{1/\b} Z_1\,.
\end{equation}
Indeed, (\ref{eqdis}) follows from $\dE \exp ( -s\sum w_i Z_{i } ) = \exp ( -
s^{\b} \sum w^\b_i ) $ and a change of variables.

\begin{lem}\label{le:stdom0}
  Assume (H2). There exists $\e>0$ and $p\in(0,1)$ such that the random
  variable $|X_{11}|^2$ dominates stochastically the random variable $\e \, D
  \,Z$, where $\dP( D = 1) = 1- \dP (D = 0) = p $ is a random variable with
  law $\mathrm{Be}(p)$ , $Z=Z^{(\b)}$ with $\b=\frac\a{2}$, and $D$ and $Z$
  are independent.
\end{lem}

\begin{proof} From our assumptions, there exist $\d>0$ and $x_0>0$ such that 
  \[
  \dP(|X_{11}|^2>x)\geq \d \, x^{-\b}\geq \dP(\d^2\,Z>x)\,,\quad x> x_0\,.
  \]
  Let $p$ be the probability that $|X_{11}|^2>x_0$. If $x>x_0$ then
  $\dP(|X_{11}|^2>x)\geq p\,\dP(\d^2\,Z>x)=\dP(\d^2\,D\,Z>x)$. On the other
  hand, if $x\leq x_0$ then $\dP(|X_{11}|^2>x)\geq p \geq \dP(\d^2\,D\,Z>x)$.
  In any case, setting $\e=\d^2$ we have
  \[
  \dP(|X_{11}|^2>x) \geq \dP(\e\,D\,Z>x)\,,\quad x>0.
  \]
  This implies the lemma.
\end{proof}

\begin{lem}\label{le:stdom}
  Assume (H2). Let $\omega_i\in[0,1]$ be numbers such that
  $\omega(n):=\sum_{i=1}^n\omega_i\geq n^{\frac12 + \e}$ for some $\e>0$. Let
  $X_1=(X_{1i})_{1\leq i\leq n}$ be i.i.d.\ random variables distributed as
  $X_{11}$, and let $Z=Z^{(\b)}$ with $\b=\frac\a{2}$. There exist $\d>0$ and
  a coupling of $X_1$ and $Z$ such that
  \begin{equation}\label{stdom1}
    \dP\left(\sum_{i=1}^n\omega_i|X_{1i}|^2\leq \d \,\omega(n)^{\frac1\b}Z\right)%
    \leq e^{-n^{\d}}\,.
  \end{equation}
\end{lem}

\begin{proof}
  Let $D=(D_i)_{1\leq i\leq n} $ denote an i.i.d.\  vector of Bernoulli
  variables with parameter $p$ given by lemma \ref{le:stdom0}. From this
  latter lemma and (\ref{eqdis}) we know that there exist $\e>0$ and a
  coupling of $X_1$, $D$ and $Z$ such
  that 
  \[
  \dP\left(\sum_{i=1}^n\omega_i|X_{1i}|^2 %
    \geq \e\Big(\sum_{i=1}^n\omega_i^\b\,D_i \Big)^{\frac1\b} Z\right)=1\,.
  \] 
  It remains to show that for some $\e'>0$:
  \[
  \dP\left(\sum_{i=1}^n\omega_i^\b\,D_i  \leq \e'\,\omega(n)\right)\leq
  e^{-n^{\e'}}\,.
  \]
  Observe that $\omega_i^\b\geq \omega_i$, so that
  $\dE\sum_{i=1}^n\omega_i^\b\,D_i \geq p\,\omega(n)$. Therefore, for $0 < \e'
  < p$,
  \begin{align*}
    &\dP\left(\sum_{i=1}^n\omega_i^\b\,D_i \leq \e'\,\omega(n)\right) \\&
    \qquad \leq \dP\left(\Big| \sum_{i=1}^n\left(\omega_i^\b\,D_i -
        \dE\omega_i^\b\,D_i \right)\Big| \geq (p-\e')\,\omega(n)\right)\leq
    2\nep{-2(p-\e')^2\omega(n)^2/n}\,,
  \end{align*}
  where we have used the Hoeffding inequality in the last bound. Since
  $\omega(n)\geq n^{\frac12 + \e}$, this implies the lemma.
\end{proof}

\begin{prop}\label{prop:distXW2}
  Assume (H2) and take $0<\gamma \leq\a/4$. Let $R$ be the first row of the
  matrix $a_n ( A - z)$. There exists a constant $c
  >0$ 
  and an event $E$ such that for any $d$-dimensional subspace $W$ of $\dC^n$
  with codimension $ n - d \geq n ^{1 - \gamma}$, we have
  \[
  \dE [   \dist^{-2} ( R, W) \, ; \,E  ]  \leq   c \,(n-d)^{-\frac 2 \alpha}    
  \quad\text{and}\quad
  \dP ( E^c ) \leq c \, n^{-(1-2\g)/\a}
  \,.
  \]
\end{prop}

\begin{proof}
  As in the proof of proposition \ref{prop:distXW}, we have
  \[
  \dist ( R , W) \geq  \dist ( X_1  , W_1)\,,
  \]
  where $W_1 = \vect ( W, e_1)$, $d \leq \dim W_1 \leq d +1$, and
  $X_1=(X_{1i})_{1\leq i\leq n}$ is the first row of $X = a_n A$. Let $\cI$
  denote the set of indexes $i$ such that $|X_{1i}|\leq a_n$. From
  (\ref{chebbo}) we know that
  \begin{equation*}
    \dP(|\cI|<n- \sqrt n )<e^{-n^{\d}}\,,
  \end{equation*}
  for some $\d>0$. It is thus sufficient to prove that for any set $I\subset
  \{1,\dots,n\}$ such that $|I|\geq n- \sqrt n $,
  \[
  \dE [ \dist^{-2} ( R, W) \, ; \,E_I \tc \cI = I ] %
  \leq c \,(n-d)^{-\frac 2 \alpha} \,,
  \]
  for some event $E_I$ satisfying $\dP ( (E_I) ^c \tc \cI=I ) \leq
  n^{-(1-2\g)/\a}$. We will then simply set
  \[
  E = E_{\cI} \cap \{|\cI| \geq n- \sqrt n\}.
  \]
  Without loss of generality, we assume that $I = \{ 1, \cdots, n' \}$ with $n '
  \geq n - n^{1/2} $. Let $\pi_I $ be the orthogonal projection on $\vect ( e_ i
  : i \in I )$. If $W_2 = \pi_I ( W_1)$, set
  \[
  W'={\rm span}\left(W_2,\dE (\pi_I(X_1)\tc \cI=I)\right)\,.
  \] 
  Note that $d - \sqrt n \leq \dim ( W') \leq \dim (W_1) +1\leq d+2$. Defining
  \[
  Y = \pi_I ( X_1)- \dE(\pi_I(X_1)\tc \cI=I),
  \]
  we have
  \[
  \dist ( R , W) \geq \dist ( X_1  , W_1) \geq \dist( Y  , W').  
  \]
  Thus, $Y=(Y_i)_{1\leq i\leq n'}$ is an i.i.d.\ mean zero vector under
  $\dP(\cdot\tc\cI=I)$. Let $P$ denote the orthogonal projection matrix to the
  orthogonal of $W'$ in $\dC^{n'}$. By construction, we have
  \begin{equation*}
    \dE \left(\dist^2 ( Y , W')\tc \cI=I\right) 
    = \dE\Big( \sum_{i,j=1}^{n'} 
    Y_i P_{ij} \bar Y_j 
    \tc \cI=I\Big)=   \dE\left[| Y_1| ^2\tc\cI=I\right]  \,\tr P   \,.
  \end{equation*}
  Here $\tr P=\sum_{i=1}^{n'} P_{ii}$, where $P_{ii}=(e_i,Pe_i)\in[0,1]$ and
  $\tr P=n'-\dim(W')$ satisfies
  \begin{equation}\label{trpbo}
    2(n-d)\geq \tr P\geq \frac12(n-d)\,.
  \end{equation}
  Let $S = \sum_{i=1}^{n'} P_{ii}  | Y_i| ^2$. We have 
  \begin{eqnarray*}
    \dE \left(  (\dist^2 ( Y, W') - S)^2 \tc \cI=I\right)  & = & \dE  \left( \Big(\sum_{i\ne j} Y_i P_{ij} \bar Y_j \Big)^2\tc \cI=I\right) \\
    & = & \sum_{ ( i_1 \neq j_1) , ( i_2 \neq j_2) }  P_{i_1 j_1} P_{i_2 j_2} \dE \left(
      Y_{i_1 } \bar Y_{j_1} Y_{i_2}\bar Y_{j_2} \tc \cI=I\right)\\
    & = & 2 \sum_{ i_1 \neq j_1 } P_{i_1 j_1} ^2 \dE[| Y_1| ^2\tc \cI=I] \\
    & \leq & 2 \dE[| Y_1| ^2\tc \cI=I] \,\tr P^2 . 
  \end{eqnarray*}
  Note that, 
  \begin{align*}
    \dE[| Y_1| ^2\tc \cI=I ] &\leq \dE[| X_{11}| ^2\tc \cI=I] \\
    &= \dE[| X_{11}|^2\tc | X_{11}|\leq a_n]\\
    &\leq \,\frac{ \dE[| X_{11}| ^2\,;\,| X_{11}|\leq a_n] }{\dP ( |
      X_{11}|\leq a_n )} =O(a_n^2/n)\,,
  \end{align*}
  where the last bound follows from lemma \ref{le:XV2}. Since $P^2 = P$, we
  deduce that 
  \begin{equation}\label{dedo}
    \dE \left[ (\dist^2 ( Y, W') - S )^2\tc \cI=I \right] %
    = O \left( a_n ^{ 2 }\,\frac{n - d}{n}\right)\,.
  \end{equation}
  Next, let $Z=Z^{(\b)}$ with $\b=\frac\a{2}$, as in lemma \ref{le:stdom}. Set
  $\omega_i = P_{ii}$, $i=1,\dots n'$, and for $\e >0$, consider the
  event 
  \[
  \G_I=\left\{ \sum_{i=1}^{n'}\omega_i|X_{1i}|^2 %
    \geq \e\,(n-d)^{\frac1\b} Z\right\}\,.
  \]
  From lemma \ref{le:stdom} (with $n$ replaced by $n'\geq n-n^{1/2}$) and
  using (\ref{trpbo}) there exists a coupling of the vector $X_{1i},
  i=1,\dots,n'$ and $Z$ such that
  \begin{equation}\label{cop1}
    \dP(\G_I^c)\leq \nep{-n^\d}\,,
  \end{equation}
  for some $\d>0$ and some choice of $\e>0$. Also, since $(a-b)^2\geq a^2/2 - b^2$ for all $a,b\in\bbR$, we have 
  $S\geq \frac12\, S_a - \,S_b$, where 
  \[
  S_a = \sum_{i=1}^{n'}\omega_i|X_{1i}|^2\,,\quad S_b = \sum_{i=1}^{n'}\omega_i\,\dE\left[|X_{1i}|\tc |X_{1i}|\leq \,a_n\right]^2\,.
  \]
  From Lemma \ref{le:XV2} and (\ref{trpbo}) we have 
  \begin{equation}\label{cop3}
    S_b = \dE\left[|X_{11}|\tc |X_{11}|\leq a_n\right]^2\,\tr P = h^{(\a)}(n,d)
 \end{equation}
 where $h^{(\a)}(n,d) \sim (n-d)a_n^2/n^2$ if $\a\in(0,1]$ and 
 $h^{(\a)}(n,d) \sim (n-d)$ if $\a\in(1,2)$.
 Let $G^1_I$ be the event that $S_a\geq 3\,S_b$. 
 From (\ref{cop3}) and the definition of $\G_I$ we have, for some $c_0 >0$
 \[
 \dP((G^1_I)^c\cap \G_I\tc \cI=I) \leq \dP(Z\leq
 c_0(n-d)^{-1/\b}h^{(\a)}(n,d)\tc \cI=I)\,.
 \]
 Note that, thanks to the assumptions $n-d\geq n^{1-\g}$, $\g\leq\a/4$, we have  
 $(n-d)^{-1/\b}h^{(\a)}(n,d)\leq  n^{-\e_0}$ for some $\e_0=\e_0(\a)>0$ for all $\a\in(0,2)$, for $n \gg 1$.
 Therefore, for $n \gg 1$, 
 \begin{align*}
   \dP((G^1_I)^c\cap \G_I\tc \cI=I)& \leq \dP(Z\leq c_0\,n^{-\e_0 }\tc \cI=I)
   \\ &= \frac{\dP(Z\leq c_0\,n^{-\e_0 }; \,\, |X_{1i}|\leq a_n \,,\;\forall
     i=1,\dots, n' )}{\dP( |X_{1i}|\leq a_n \,,\;\forall i=1,\dots, n' )} \,,
\end{align*}
where the last identity follows from the independence of the $X_{1i}$.
Observing that the probability for the event $\{|X_{1i}|\leq a_n \,,\;\forall
i=1,\dots, n' \}$ is lower bounded by $1/c >0$ uniformly in $n$, we obtain
  \[
 \dP((G^1_I)^c\cap \G_I\tc \cI=I)\leq  c\,\dP(Z\leq c_0\,n^{-\e_0} )\,.
 \]
 The latter probability can be estimated using Markov's inequality and the
 fact that $\dE[Z^{-m}]=u_m$ is finite (cf. (\ref{um})). Indeed, for every
 $m>0$, $ \dP(Z\leq t )\leq \,u_m t ^{-m} \,. $ Thus, we have shown that for
 every $p>0$ there exists a constant $\kappa_p$ such that
 \begin{equation}\label{cop4}
 \dP((G^1_I)^c\cap \G_I\tc \cI=I)\leq \kappa_p\,n^{-p}\,.
\end{equation}
Next, we set $\wt\G_I=G^1_I\cap \G_I$ and we claim that 
\begin{equation}\label{cop2}
 \dE   \left[S^{-2} \,;\,\wt\G_I\tc \cI=I \right] = O \left( \,(n-d)^{-4/\a}\,\right) ,
 \end{equation}
Indeed, on $\wt\G_I$ we have 
 $S\geq \frac16\,S_a\geq \frac\e6 (n-d)^{2/\a}\, Z$ and therefore, for some constant $c_1$, 
 \[
  \dE   \left[S^{-2} \,;\,\wt\G_I\tc \cI=I \right]\leq  c_1\,(n-d)^{-4/\a}
  \dE   \left[Z^{-2} \tc \cI=I \right]\,.
 \]
Using independence as before, and recalling that  
the event $\{|X_{1i}|\leq a_n \,,\;\forall i=1,\dots, n' \}$ has uniformly positive probability we have
\[
 \dE   \left[Z^{-2} \tc \cI=I \right]\leq c\,\dE[Z^{-2}]= c\, u_2\,.
\]
This proves (\ref{cop2}).

Now, for the event   Markov's and Cauchy-Schwarz' inequalities  lead to 
\begin{eqnarray*}
\dP \left(  \dist^2 ( Y, W')  \leq S / 2  \,;\,\wt\G_I\tc \cI=I \right)  & \leq & \dP  \left(\frac{|\dist^2 ( Y, W') - S| }{S} \geq 1/2\,;\,\wt\G_I\tc \cI=I \right) \\
& \leq & 2 \dE  \left[\frac{|\dist^2 ( Y, W') - S| }{S} \,\,;\,\wt\G_I\tc \cI=I \right] \\
& \leq &  2 \sqrt{\dE  \left[  |\dist^2 ( Y, W') - S|^2\tc \cI=I \right]   \dE   \left[S^{-2} \,;\,\wt\G_I\tc \cI=I \right]}.
\end{eqnarray*}
  Hence, if $G^2_I$ denotes the event $\{ \dist^2 ( Y, W') \geq S / 2 \}$, we
  deduce from (\ref{dedo}) and (\ref{cop2})
\begin{equation}\label{cop5}
\dP \left(  (G^2_I)^c\cap\wt\G_I\tc \cI=I \right) =  O \left( \,a_n n^{-\frac12}(n - d)^{\frac12-\frac{2}\a}\,\right) .
\end{equation} 
Note that, using $n-d\geq n^{1-\g}$,  the last expression is certainly $O(n^{-\frac1\a\,(1-2\g)})$. On the other hand, by (\ref{cop2}) and Cauchy-Schwarz' inequality
\begin{equation}\label{cop6}
\dE  \left[  \dist^{-2} ( X, W)  \,;\,G^2_I\cap \wt\G_I \tc \cI=I\right]
 \leq 2  \,\dE \left[S^{-1} \,;\,\wt\G_I \tc \cI=I\right]  = O \left(  \,(n-d)^{-2/\alpha}   \,\right) .
\end{equation}
To conclude the proof we take $E_I = G^2_I\cap \wt\G_I = G^1_I\cap G^2_I\cap \G_I$.   
We have
\[
\dP((E_I)^c\tc \cI=I) \leq \dP\left((\G_I)^c\tc\cI=I\right) +\dP\left ((G^1_I)^c\cap \G_I\tc \cI=I\right)
 +\dP \left(  (G^2_I)^c\cap G^1_I\cap \G_I\tc \cI=I \right)\,.
\] 
From (\ref{cop4}) and (\ref{cop5}) we see that, 
\[
\dP\left ((G^1_I)^c\cap \G_I\tc \cI=I\right)
 +\dP \left(  (G^2_I)^c\cap G^1_I\cap \G_I\tc \cI=I \right) = O \left(\,  n^{-\frac1\a\,(1-2\g)}\,\right),
\]
and all it remains to prove is an upper bound on $\dP\left((\G_I)^c\tc\cI=I\right)$. By independence, as before
\[\dP\left((\G_I)^c\tc\cI=I\right)\leq c \,\dP\left((\G_I)^c\,;\, |X_{1i}|\leq a_n \,,\;\forall i=1,\dots, n'\right)\,.\]
From (\ref{cop1}) we obtain 
$\dP\left((\G_I)^c\tc\cI=I\right)\leq c\,\nep{-n^\d}$. 
This ends the proof.
\end{proof}

\subsection{Uniform integrability}

Let $z \in \dC$ and $\sigma_n \leq \cdots \leq \sigma_1$ be the singular
values of $A_n- z$ with $A_n$ defined by \eqref{eq:defAn}. For $0 < \delta
<1$, we define $K_\delta = [\delta , \delta^{-1}]$. In this paragraph, we
prove the uniform integrability in probability, meaning that for all $\e >0$,
there exists $\delta >0$ such that
\begin{equation}\label{eq:unifconvproba}
  \dP \left(  \int_{K_\delta^c} | \ln (x) | \nu_{A_n - z} (dx)  > \e  \right) \to 0.
\end{equation}
From lemma \ref{le:largeSV}, with probability $1$ there exists $c_0 >0$, such
that for all $n$,
\[ 
\int_{1}^\infty \ln^2  (x)  \nu_{A_n - z} (dx) < c_0.
\]
It follows from Markov inequality that for all $t \geq 1$, $ \int_{t}^\infty
\ln (x) \nu_{A_n - z } (dx) < c_0 / \ln t $. The upper part $(\delta^{-1} ,
\infty)$ of \eqref{eq:unifconvproba} is thus not an issue. For the lower part
$(0, \delta)$, it is sufficient to prove that
\[
\frac 1 n \sum_{i=0}^{n-1} \ind_{\{\sigma_{n-i} \leq \delta_n\}} \ln
\sigma^{-2}_{n-i}
\]
converges in probability to $0$ for any sequence $(\delta_n)_n$ converging to
$0$. From lemma \ref{le:smallestSV}, we may a.s. lower bound $\sigma_{n - i}$
by $c n^{-r}$ for some constant $c$ and all integer $n \geq 1$. Take $0
<\gamma < \a/4$ to be fixed later. Using this latter bound for every $1 \leq i
\leq n^{1 - \gamma}$, it follows that it is sufficient to prove that
\[
\frac 1 n \sum_{i = \lfloor n^{1 - \gamma} \rfloor } ^{n-1 }
\ind_{\{\sigma_{n-i} \leq \delta_n\}} \ln \sigma^{-2}_{n-i}
\]
converges in probability to $0$. We are going to prove that there exists an
event $F_n$ such that, for some $\delta >0$ and $c >0$,
\begin{equation}\label{eq:Fn}
  \dP ( (F_n)^c ) \leq  c \exp ( - n^\delta ),
\end{equation}
and 
\begin{equation}\label{eq:sigmani}
  \dE \left[\sigma_{n-i} ^{-2} \tc F_n \right] \leq c  \left( \frac n i \right)^{ \frac 2 \alpha + 1}  .
\end{equation}
We first conclude the proof before proving \eqref{eq:Fn}-\eqref{eq:sigmani}.
From Markov inequality, and \eqref{eq:sigmani}, we deduce that
\[
\dP ( \sigma_{n-i} \leq \delta_n ) \leq \dP ((F_n)^c) + c\, \delta_n ^2 \left( \frac n i \right)^{ \frac 2 \alpha + 1}. 
\]
If follows that there exists a sequence $\e _n = \delta_n ^{ 1/ ( \frac 2
  \alpha + 1)} $ tending to $0$ such that the probability that $\dP (
\sigma_{n- \lfloor n \e_n \rfloor} \leq \delta_n )$ converges to $0$. We
obtain that it is sufficient to prove that
\[
\frac 1 n \sum_{i = \lfloor  n^{1 - \gamma} \rfloor } ^{\lfloor  \e_n n  \rfloor }   \ln  \sigma^{-2}_{n-i} 
\]
given $F_n$ converges in probability to $0$. However, using the concavity of
the logarithm and \eqref{eq:sigmani} we have
\begin{eqnarray*}
  \dE \left[ \frac 1 n \sum_{i = \lfloor  n^{1 - \gamma} \rfloor } ^{ \lfloor  \e_n n  \rfloor }   \ln  \sigma^{-2}_{n-i} \Bigm| F_n  \right] & \leq &  \frac 1 n \sum_{i = \lfloor  n^{1 - \gamma} \rfloor } ^{ \lfloor  \e_n n  \rfloor }    \ln  \dE [   \sigma^{-2}_{n-i}  | F_n  ]\\
  & \leq & \frac{c_1 }{ n} \sum_{i =1} ^{\lfloor  \e_n n  \rfloor }    \ln  \left( \frac n i \right)  \\
  & = & c_1 \left(  -\e_n \ln \e_n  + \e_n + O ( n ^{-1}) \right). 
\end{eqnarray*}
It thus remain to prove \eqref{eq:Fn}-\eqref{eq:sigmani}. Let $B_n$ be the
matrix formed by the first $n-\lfloor i /2 \rfloor $ rows of $a_n ( A_n - z I
)$. If $\sigma'_{1} \geq \cdots \geq \sigma'_{n-\lfloor i /2 \rfloor }$ are
the singular values of $B_n$, then by the Cauchy interlacing Lemma
\ref{le:cauchy},
\[
\sigma_{n-i} \geq \frac{ \sigma'_{n-i} }{ a_n}. 
\]
By the Tao-Vu negative second moment lemma \ref{le:tvneg}, we have
\[
\sigma^{'-2}_1 + \cdots +\sigma^{'-2}_{n - \lceil i /2 \rceil} = \dist^{-2}_1 + \cdots + \dist^{-2}_{n - \lceil i /2 \rceil},
\]
where $\dist_j$ is the distance from the $j$-th row of $B_n$ to the subspace spanned by the other rows of $B_n$. In particular, 
\[
\frac i 2 \sigma^{-2}_{n-i} %
\leq a^2_n \sum_{j=1} ^ {n - \lfloor i /2 \rfloor} \dist^{-2}_j.
\]
Let $F_n$ be the event that for all $1 \leq j \leq n - \lfloor i /2 \rfloor$,
$ \dist_j \geq n ^{(1-2\g)/\a }$. Since the dimension of the span of all but
one rows of $B_n$ is at most $d\leq n-i/2$, we can use proposition
\ref{prop:distXW}, to obtain
\[
\dP( (F_n)^c) \leq  \exp ( - n ^\delta )\,,
\] 
for some $\delta >0$. Then we write
\[
\frac i 2  \sigma^{-2}_{n-i}
\,\ind_{F_n} 
\leq a^2_n \,\sum_{j=1} ^ {n - \lfloor i /2 \rfloor}
\dist^{-2}_j\,   \, \ind_{F_n}  \,,
\] 
Taking expectation, we get 
\begin{equation}\label{sumi}
  \dE\left[ i   \sigma^{-2}_{n-i} \,;\,F_n\right] 
  \leq   2\,a^{2}_n  n  \dE\left[\dist^{-2}_1\,   ; \,  F_n \right]\,,
\end{equation}
Since we are on $F_n$ we can always estimate $ \dist_1 \geq n ^{(1-2\g)/\a} $.
By introducing a further decomposition we can strengthen this as follows.
Recall that from proposition \ref{prop:distXW2}, there exists an event $E$
independent from the rows $j\neq 1$ such that $\dP((E)^c )\leq
n^{-(1-2\g)/\a}$ and for any $W\subset \dC^n$ with dimension $d<n-n^{1-\g}$
one has
\[
\dE [ \dist(R,W)^{-2} \,;\,E ]\leq c \,(n-d)^{-2/\a}\,.
\]
Here $R$ is the first row of the matrix $B_n$. By first conditioning
on the value of the other rows of $B_n$ and recalling that the dimension $d$
of the span of these is at most $n-i/2\leq n-2n^{1-\g}$, we see that
\[
\dE[ \dist^{-2}_1 \,;\,E ] = O\left( \,i^{-2/\a}\,\right).  
\]
Therefore
\begin{align}\label{II}
  \dE\left[\dist^{-2}_1 \, ; \, F_n \right]
  &\leq  \dE(\dist^{-2}_1 \,;\,E ) + \dP((E)^c) \,n^{-2(1-2\g)/\a} \nonumber\\
  &\leq c_2\, \left( i^{-2/\a} + n^{-3(1-2\g)/\a}\right)\,.
\end{align}
Now, if $\g<1/6$ we have $3(1-2\g)/\a>2/\a$ and therefore
$n^{-3(1-2\g)/\a}\leq i^{-2/\a}$. Thus, (\ref{II}) implies
\begin{equation}\label{III}
  \dE\left[\dist^{-2}_1   \, ; \, F_n \right]
  \leq 2\,c_2\, 
  i^{-2/\a} \,.
\end{equation}
From (\ref{sumi}) we obtain
\[
\dE\left[ i \sigma^{-2}_{n-i} \,;\,F_n\right] \leq 2\,c_2 \,a^{2}_n \, n \,
i^{-2/\a} \,.\] From (H2) it follows that \eqref{eq:sigmani} holds. This
concludes the proof of \eqref{eq:Fn}-\eqref{eq:sigmani}.

\subsection{Proof of theorem \ref{th:girko}}\label{ss:girko}
 
We may now invoke theorem \ref{th:mpz} and \eqref{eq:unifconvproba}. From
lemma \ref{le:girko}, $\mu_{A_n}$ converges in probability to $\mu_\a$, where
for almost all $z \in \dC$, 
\[
U_{\mu_\a} (z) = \int \ln( x) \nu_{\alpha,z}(dx).
\] 

Let us upgrade this convergence to an a.s.\ convergence. By lemmas
\ref{le:largeSV} and \ref{le:unicity}, it is sufficient to prove that for
every $z\in\dC$, a.s.
\[
\lim_{n \to \infty} U_{\mu_{A_n}} (z)=  U_{\mu_a}(z).
\]
Let us fix $z\in\dC$ from now on. Since $L=U_{\mu_a}(z)$ is
\emph{deterministic}, it actually suffices to show that there exists a
\emph{deterministic} sequence $L_n$ such that a.s.
\begin{equation}\label{eq:ULn}
\lim_{n \to \infty} \left( U_{\mu_{A_n}} (z) -  L_n \right) = 0.
\end{equation}
Now, by lemmas \ref{le:largeSV} and \ref{le:smallestSV}, there exists $b>0$
such that a.s. for $n \gg 1$,
\[
\supp(\nu_{A_n-zI})\subset [s_n(A_n-zI), s_1(A_n-zI)]\subset[n^{-b},n^b]. 
\] 
Denoting $f_n:x\in\dR_+\mapsto f_n (x) = \mathbf{1}_{[n^{-b},n^b]}(x)\log(x)$,
we get that a.s. for $n\gg1$,
\begin{equation}\label{eq:Ufn}
U_{\mu_{A_n}} (z) %
= -\int_0^\infty\!\log(s)\,d\nu_{A_n-zI}(s) %
= -\int_0^\infty\!f_n(s)\,d\nu_{A_n-zI}(s).
\end{equation}
The total variation of $f_n$ is bounded by $c \log n$ for some $c >0$. Hence
by lemma \ref{le:concspec}, if
\[
L_n := \dE \int f_n(s) d \nu_{A_n-zI} (s),
\]
then we have, for every $\varepsilon>0$,
\[
\dP\PAR{\ABS{ \int  f_n(s) d \nu_{A_n-zI} (s) - L_n} \geq \varepsilon} %
\leq 2 \exp\PAR{- 2 \frac{ n \varepsilon^2 } { (c \log n)^2 } }.
\]
In particular, from the first Borel-Cantelli lemma, a.s.,
\[
\lim_{n \to \infty} \PAR{\int\!f_n(s)\,d\nu_{A_n-zI}(s)-L_n}= 0.
\]
Finally, using \eqref{eq:Ufn}, we deduce that \eqref{eq:ULn} holds almost
surely, as required.

\section{Limiting spectral measure}\label{se:mu}

In this section, we give a close look to the resolvent of the random operator
on the PWIT and we deduce some properties of the limiting spectral measure
$\mu_\alpha$. For ease of notation we set
\[
\beta = \frac \alpha 2 
\]
and define the measure on $\dR_+$, 
\[
\Lambda_\alpha = \frac{\a}{2} x ^{-\frac \a 2 -1}dx.
\]

\subsection{Resolvent operator on the Poisson Weighted Infinite Tree}

In this paragraph, we analyze the random variable 
\[
R(U)_{\o \o} = 
\begin{pmatrix} 
 a (z,\eta) & b(z,\eta)  \\  
 b' (z,\eta) &  c (z,\eta) 
\end{pmatrix}.
\]
By lemma \ref{le:Rbounded}, for $t \in \dR_+$, $a (z, it) $ is pure imaginary
and we set
\[
h(z,t) = \Im  ( a (z, it) )  = - i a (z, it) \in [0, t^{-1}]. 
\]
The random variables $a(z,\eta)$ and $h(z,t)$ solve a nice
recursive distribution equation.

\begin{thm}[Recursive Distributional Equation] \label{th:BC} Let $U = U (z,
  \eta) \in \dH_+$, $t \in \dR_+$. Let $L_U$ be the distribution on $\dC_+$ of
  $a(z,\eta)$ and $L_{z,t}$ the distribution of $h(z,t)$.
  \begin{enumerate}
  \item[(i)]
    $L_U$ solves the equation in distribution 
    \begin{equation} \label{eq:RDEa} %
      a \overset{d}{=} \frac{ \eta + \sum_{k \in \dN} \xi_k a_k }{ | z |^2 -
        \left( \eta + \sum_{k \in \dN} \xi_k a_k \right) \left( \eta + \sum_{k
            \in \dN} \xi'_k a'_k \right) },
    \end{equation} 
    where $a$, $(a_k)_{k \in \dN}$ and $(a'_k)_{k \in \dN}$ are i.i.d.\  with
    law $L_U$ independent of $\{\xi_k\}_{k \in \dN}$, $\{\xi'_k\}_{k \in \dN}$
    two independent Poisson point processes on $\dR_+$ with intensity
    $\Lambda_\alpha$.
  \item[(ii)] $L_{z,t}$ is the unique probability distribution on $[0,\infty)$
    such that
    \begin{equation} \label{eq:RDEh} %
      h \overset{d}{=} \frac{ t + \sum_{k \in \dN} \xi_k h_k }{ | z |^2 +
        \left( t + \sum_{k \in \dN} \xi_k h_k \right) \left( t + \sum_{k \in
            \dN} \xi'_k h'_k \right) }
    \end{equation} 
    where $h$, $(h_k)_{k \in \dN}$ and $(h'_k)_{k \in \dN}$ are i.i.d.\ with
    law $L_{z,t}$, independent of $\{\xi_k\}_{k \in \dN}$, $\{\xi'_k\}_{k \in
      \dN}$ two independent Poisson point processes on $\dR_+$ with intensity
    $\Lambda_\alpha$.
  \item[(iii)] For $t = 0$ there are two probability distributions on
    $[0,\infty)$ solving \eqref{eq:RDEh} such that $\dE h^{\alpha/2} <
    \infty$: $\delta_0$ and another denoted by $L_{z,0}$. Moreover, for the
    topology of weak convergence, $L_{z,t}$ converges to $L_{z,0}$ as $t$ goes
    to $0$.
  \end{enumerate}
 \end{thm}

 We start with an important lemma.
\begin{lem}\label{le:schurPWIT}
  For every $U = U (z, \eta) \in \dH_+$, $\begin{pmatrix} a & b \\ b' &
    c \end{pmatrix}$ is equal in distribution to
  \begin{eqnarray}
    \frac{ 1}{  |z|^2 -  \left( \eta +  \sum_{k \in \dN}  \xi_k a_k
      \right)\left( \eta +  \sum_{k \in \dN}  \xi'_k a'_k \right)  }  
    \begin{pmatrix}  
      \eta +  \sum_{k \in \dN}  \xi_k a_k     &  - z  \\
      -  \bar z   & \eta +  \sum_{k \in \dN}   \xi'_k a'_k  
    \end{pmatrix}, \label{eq:RDE22}
  \end{eqnarray}
  where $a$, $(a_k)_{k \in \dN}$ and $(a'_k)_{k \in \dN}$ are i.i.d.\  with law
  $L_U$ independent of $\{\xi_k\}_{k \in \dN}$, $\{\xi'_k\}_{k \in \dN}$ two
  independent Poisson point processes on $\dR_+$ with intensity
  $\Lambda_\alpha$.
\end{lem}

\begin{proof}[Proof of lemma \ref{le:schurPWIT}]
  Consider a realization of $\pwit(2 \ell_\theta)$ on the tree $T$. For $k \in
  \dN$, we define $T_k$ as the subtree of $T$ spanned by $k \dN^f$. With the
  notation of lemma \ref{le:schurB}, for $k \in \dN$, $R_{B_k} (U) = (B_k(z) -
  \eta ) ^{-1}$ is the resolvent operator of $B_k$ and set
  \[
  \wt R (U)_{kk} = \Pi_k R_{B_k} (U) \Pi^*_k = \begin{pmatrix} a_k & b_k \\
    b'_k & c_k\end{pmatrix}.
  \]
  Then, by lemma \ref{le:schurB} and \eqref{eq:wApwit}, we get
  \begin{eqnarray*}
    R(U) _{\o\o}   &=&  - \left( U  + \sum_{k \in \dN}   \begin{pmatrix}    0      &   \e_k y_{k}^{-1/\alpha }     \\      (1- \e_k ) y_{k}^{-1/\alpha }  & 0   \end{pmatrix}  \begin{pmatrix} a_k & b_k  \\  b'_k &  c_k\end{pmatrix}   \begin{pmatrix}    0      &(1- \e_k )  y_{k}^{-1/\alpha }     \\     \e_k  y_{k}^{-1/\alpha }  & 0   \end{pmatrix} \right)^{-1} \\
    & =& - \left( U  +     \begin{pmatrix} \sum_{k \in \dN} (1 - \e_k ) |y_{k}|^{-2/\alpha } c_k  &  0 \\  0 &   \sum_{k \in \dN} \e_k |y_{k}|^{-2/\alpha } a_k \end{pmatrix}     \right)^{-1} \\
    & = & D^{-1} \begin{pmatrix}  \eta +  \sum_{k \in \dN}  \e_k |y_{k}|^{-2/\alpha }  a_k     &  - z  \\  -  \bar z   &\eta +  \sum_{k \in \dN} (1- \e_k )|y_{k}|^{-2/\alpha } c_k  \end{pmatrix},
  \end{eqnarray*}
  with $ D = |z|^2 - \left( \eta + \sum_{k \in \dN} \e_k |y_{k}|^{-2/\alpha }
    a_k \right)\left( \eta + \sum_{k \in \dN} (1-\e_k) |y_{k}|^{-2/\alpha }
    c_k \right) $.

  Now the structure of the PWIT implies that (i) $a_k$ and $c_k$ have common
  distribution $L_U$; and (ii) the variables $(a_k,c_k)_{k \in \dN}$ are
  i.i.d.. Also the thinning property of Poisson processes implies that (iii)
  $\{ \e_k |y_{k}|^{-2/\alpha } \}_{k \in \dN}$ and $\{ (1-\e_k)|
  y_{k}|^{-2/\alpha } \}_{k \in \dN}$ are independent Poisson point process with
  common intensity $\Lambda_\alpha$. 
\end{proof}

The next well-known and beautiful lemma will be crucial in the computations
that will follow. It is a consequence of the LePage-Woodroofe-Zinn
representation of stable laws \cite{lepagewoodroofe}, see also Panchenko and
Talagrand \cite[Lemma 2.1]{panchenko-talagrand}.

\begin{lem}\label{le:magic}
  Let $\{ \xi_k\}_{k \in \dN}$ be a Poisson process with intensity
  $\Lambda_\alpha$. If $(Y_k)$ is an i.i.d.\  sequence of non--negative random
  variables, independent of $\{ \xi_k\}_{k \in \dN}$, such that $\dE [ Y_1
  ^\beta] < \infty$ then
  \[
  \sum_{k \in \dN} \xi_k Y_k \stackrel{d}{=} \dE [ Y_1 ^\beta ]^{\frac 1
    \beta} \sum_{k \in \dN} \xi_k \stackrel{d}{=} \dE [ Y_1 ^\beta ]^{\frac 1
    \beta} S,
  \]
  where $S$ is the positive $\beta$-stable random variable with Laplace
  transform for all $x \geq 0$,
  \begin{equation}\label{eq:laplaceZ}
    \dE \exp ( - x S) =   \exp\left( - \Gamma(1-\beta) x^\b \right).
  \end{equation}
\end{lem}

\begin{proof}[Proof of lemma \ref{le:magic}] Recall the formulas, for $y \geq
  0$, $\eta >0$ and $0 < \eta < 1$ respectively,
  \begin{equation} \label{eq:gammaLaplace} %
    y^{-\eta} %
    = \Gamma(\eta)^{-1} \int_0 ^\infty x ^{\eta -1} e^{- x y} dx %
    \quad\text{and}\quad %
    y^{\eta} %
    = \Gamma(1-\eta)^{-1} \eta \int_0 ^\infty x ^{-\eta -1} (1 - e^{- x y}) dx.
  \end{equation} 
  From the L\'evy-Khinchin formula we deduce that, with $s \geq 0$,
  \begin{align*}
    \dE \exp\left( - s \sum_{k} \xi_k Y_k \right)
    & = \exp \left( \dE \int_0 ^\infty ( e^{-x  s Y_1} - 1 ) \beta x ^{-\beta - 1} dx \right) \nonumber \\
    &= \exp\left( - \Gamma(1-\beta) s^\b \dE
      [Y_1^\beta]\right) \label{eq:laplace}.
  \end{align*}
\end{proof}

\begin{proof}[Proof of theorem \ref{th:BC}] 
  Statement $(i)$ is contained in lemma \ref{le:schurPWIT}. For $(ii)$, let $t
  >0$ and $h$ a solution of \eqref{eq:RDEh}. Then $h$ is positive and is upper
  bounded by $1/t$. By lemma \ref{le:magic}, we may rewrite \eqref{eq:RDEh} as
  \begin{equation}
    \label{eq:RDEhy}
    h  \overset{d}{=}  \frac{  t  + \dE[ h^\beta] ^{1 / \beta} S    }{ |  z |^2 + \left( t  + \dE[ h^\beta] ^{1 / \beta} S    \right) \left( t  +  \dE[ h^\beta] ^{1 / \beta} S'    \right)  } 
  \end{equation}
  where $S$ and $S'$ are i.i.d.\ variables with common Laplace transform
  \eqref{eq:laplaceZ}. In particular, $\dE[ h^\beta]^{1/\beta}$ is solution of
  the equation in $y$:
  \[
  y^\beta = \dE \left( \frac{ t + y S }{ | z |^2 + \left( t + y S \right)
      \left( t + y S' \right) } \right)^\beta
  \]
  Since $t >0$, $\dE[ h^\beta] >0$, it follows that $\dE[ h^\beta]^{1/\beta}$
  is solution of the equation in $y$:
  \begin{equation}
    \label{eq:fixpty0}
    1 = \dE \left(  \frac{  t y  ^{ - 1 }    +  S    }{ |  z |^2 + \left( t  +  y  S    \right) \left( t  +  y S'    \right)  }   \right)^\beta . 
  \end{equation}
  For every $S,S' > 0$, the function $y \mapsto \frac{ t y ^{ - 1} + S }{ | z
    |^2 + \left( t + y S \right) \left( t + y S' \right) }$ is decreasing in
  $y$. It follows that
  \[
  y \mapsto \dE \left( \frac{ t y ^{ - 1 } + S }{ | z |^2 + \left( t + y S
      \right) \left( t + y S' \right) } \right)^\beta
  \]
  is decreasing in $y$. As $y$ goes to $0$ it converges to $\infty$ and as $y$
  goes to infinity, it converges to $0$. In particular, there is a unique
  point, $y_* ( |z|^2, t) $ of such that \eqref{eq:fixpty0} holds. This proves
  $(ii)$ since from \eqref{eq:RDEhy}, the law of $h$ is determined by $\dE[
  h^\beta]^{1/\beta} = y_* ( |z|^2, t)$.

  For Statement $(iii)$ and $t = 0$, then $h = 0$ is a particular solution of
  \eqref{eq:RDEh}. If $h$ is not a.s. equal to $0$, then $\dE[
  h^\beta]^{1/\beta} >0$ and the argument above still works since, for every
  $s,s' > 0$, the function $y \mapsto \frac{ s }{ | z |^2 + y ^{2 } s s' }$ is
  decreasing in $y$. We deduce the existence of a unique positive solution
  $y_* ( |z|^2, 0) $ of \eqref{eq:fixpty0}. We also have the continuity of the
  function $t \mapsto y_* ( |z|^2, t) $ on $[0,\infty)$. Finally
  \[
  h \overset{d}{=} y_* ( |z|^2, 0) S / ( | z |^2 + y^2_* ( |z|^2, 0) S S' ),
  \]
  and from \eqref{eq:RDEhy}, it implies the weak convergence of $L_{z,t}$ to
  $L_{z,0}$.
\end{proof}

\subsection{Density of the limiting measure}

In this paragraph, we analyze the RDE \eqref{eq:RDE22}. For all $t >0$, let
$L_{z,t}$ be as in theorem \ref{th:BC}. From Equation \eqref{eq:RDEhy}, $h$
may be expressed as
\begin{equation*}
  h  %
  \overset{d}{=}  %
  \frac{  t  + y_* S    }{ |  z |^2 + \left( t  + y_* S    \right) \left( t  +  y_* S'    \right)  } 
\end{equation*}
where $S$ and $S'$ are i.i.d.\ variables with common Laplace transform
\eqref{eq:laplaceZ} and $y_* := y_* ( |z|^2 , t)$ is the unique solution in
$(0,\infty)$ of \eqref{eq:fixpty0} (uniqueness is proved in theorem
\ref{th:BC}). We extend continuously the function $y_*(r,t)$ for $t = 0$ by
defining $y_*(|z|^2,0)$ as the unique solution in $(0,\infty)$:
\begin{equation}\label{eq:fixpty}
 1 = \dE \left(  \frac{     S    }{ |  z |^2 +  y^2  S S'     }   \right)^\beta . 
\end{equation}
\begin{lem}
\label{le:yC1}
The function $y_* : [0,\infty)^2 \to (0,\infty)$ is $C^1$. For every $t \geq
0$, the mapping $r \mapsto y_* ( r , t)$ is decreasing to $0$.
\end{lem}

\begin{proof}
  For every $t \geq 0$, the derivative in $y>0$ of the function $\dE \left(
    \frac{ t y ^{ - 1 } + S }{ | z |^2 + \left( t + y S \right) \left( t + y
        S' \right) } \right)^\beta $ is
\begin{equation}\label{eq:diffy}
  - \beta t y^{-2} \dE\frac{ (  t y  ^{ - 1 }    +  S ) ^{\beta-1}    }
  {  \left( |  z |^2 + \left( t  +  y  S    \right) 
      \left( t  +  y S'    \right)  \right)^\beta }     
  -  \beta \dE  \frac{( t y^{-1} + S)^\beta  ( S ( t + y S') + S' ( t + yS ) )
  }
  { \left(  |  z |^2 + \left( t  +  y  S    \right) 
      \left( t  +  y S'    \right)     \right)^{\beta +1} }.
\end{equation}
The last computation is justified since all terms are integrable, indeed we have
\[
\frac{ ( t y ^{ - 1 } + S ) ^{\beta-1} }{ \left( | z |^2 + \left( t + y S
    \right) \left( t + y S' \right) \right)^\beta } \leq \frac{ y^{-\beta+1}
}{ \left( t + y S \right) \left( t + y S' \right)^\beta } \leq \frac{
  y^{-2\beta} }{ S {S'} ^\beta }
\]
and from \eqref{eq:gammaLaplace}, for all $\eta >0$,
\begin{equation}\label{eq:Seta}
\dE S^{-\eta} =   \Gamma (\eta) ^{-1}  \int x^{\eta - 1} e^{ - \Gamma ( 1- \beta) x^{\beta}}  dx < \infty. 
\end{equation} 
Similarly, for the second term of \eqref{eq:diffy}, we write
\begin{eqnarray*}
  \frac{( t y^{-1} + S)^\beta  ( S ( t + y S') + S' ( t + yS ) )  }{ \left(  |  z |^2 + \left( t  +  y  S    \right) \left( t  +  y S'    \right)     \right)^{\beta +1} } & \leq &   y^{-1} \frac{  S ( t + y S') + S' ( t + yS )  }{  \left( t  +  y  S    \right) \left( t  +  y S'    \right)^{\beta +1} } \\
  & \leq & y^{-1} \frac{  S   }{  \left( t  +  y  S    \right) \left( t  +  y S'    \right)^{\beta} } + y^{-1} \frac{   S'    }{   \left( t  +  y S'    \right)^{\beta +1} } \\
  & \leq & y^{-\beta - 2}    {S'} ^{-\beta}  + y^{-\beta - 2 }  {S'}^{-\beta }  
   \end{eqnarray*}
   The expression \eqref{eq:diffy} is finite and strictly negative for all $y
   >0$. The statement follows from the implicit function theorem.
\end{proof}

From \eqref{eq:RDE22}, for all $ t > 0$,
\[
b ( z , it ) \overset{d}{=} - \frac{ z }{| z |^2 + \left( t + y_* ( |z|^2 ,
    it) S \right) \left( t + y_* ( |z|^2 , it)S' \right) }.
\] 
By lemma \ref{le:yC1}, we may also define
\[
b ( z , 0 ) = \lim_{ t \downarrow 0 } b ( z , it ) \overset{d}{=} - \frac{ z
}{| z |^2 + y^2_* ( |z|^2 , 0) S S' }.
\]
For ease of notation, we set $y_* ( r) = y_* ( r , 0)$. Since $\partial z =
1$, $\partial |z|^2 = \bar z$, we deduce that
\begin{eqnarray}
  - \dE \partial b(z,0)  & = &    \dE  \partial  \frac{  z   }{|  z |^2 +   y^2 _* ( |z|^2)  S  S' } \nonumber  \\
  & =&   \dE  \left(|  z |^2 +   y^2 _* ( |z|^2)  S  S'  \right)^{-1} -|z|^{2}   \dE  \left(|  z |^2 +   y^2 _* ( |z|^2)  S  S'  \right)^{-2}  \nonumber \\
  & & \quad\quad\quad \quad \quad  -   2 |z|^{2}  y_* ( | z|^2 ) y'_* ( | z|^2 )  \dE  S S' \left(|  z |^2 +   y^2 _* ( |z|^2)  S  S'  \right)^{-2} \nonumber \\
  & = & \left(y^2 _* ( |z|^2)  - 2  |z|^{2} y _* ( |z|^2)  y' _* ( |z|^2) \right) \dE \frac {S S'}{ \left(|  z |^2 +   y^2 _* ( |z|^2)  S  S'  \right)^2}   \label{eq:density} .
\end{eqnarray}
The latter is justified since 
\[
S S' \left(|  z |^2 +   y^2   S  S'  \right)^{-2} \leq y^{-4} (S  S') ^{-1}. 
\]
is integrable from \eqref{eq:Seta}.
The next lemma is an important consequence of Theorems \ref{th:convRes} and
\ref{th:girko}.

\begin{lem}
  \label{le:density}
  The following identity holds in $\cD' (\dC)$:
  \[
  \mu_\a = - \frac {1}{ \pi} \partial \dE b ( \cdot,0 ) .
  \]
  Therefore the measure $\mu_\a$ is isotropic and has a continuous density
  given by $1/\pi$ times the right hand side of \eqref{eq:density}.
\end{lem}

\begin{proof}
  Let $R_n$ be the resolvent matrix of $B_n$, the bipartized matrix of $A_n$
  defined by \eqref{eq:defAn}. By theorem \ref{th:convRes} and lemma
  \ref{le:Rbounded}, for all $t >0$ and $z \in \dC$,
  \[
  \lim_{n\to\infty} \dE R_n (U(z,it))_{11} = \begin{pmatrix} i \dE h (z,t) &
    \dE b (z,it) \\ \dE \bar b(z,it) & i \dE h (z,t) \end{pmatrix},
  \]
  From theorem \ref{th:convSV}, $\dE \nu_{A_n - z}$ converge weakly to
  $\nu_{\alpha,z}$ and, by lemma \ref{le:largeSV}, for all $t >0$,
  \[
  \lim_{n\to\infty} \frac 1 2 \int \ln ( x^2 + t^2 ) \dE \nu_{A_n - z } dx) =
  \frac 1 2 \int \ln ( x^2 + t^2) \nu_{\alpha,z} (dx).
  \]
  From Equation \eqref{eq:unifconvproba}, $\int \ln (x) \nu_{\alpha,z}(dx)$ is
  integrable. We deduce that for all $z_0 \in \dC$, there exists an open
  neighborhood of $z_0$ and a sequence $(t_n)_{n \geq 1}$ converging to $0$
  such that for all $z$ in the neighborhood,
  \begin{equation}\label{eq:Restn}
    \lim_{n\to\infty} \dE   R_n (U(z,it_n))_{11}  = \begin{pmatrix} i \dE h(z,0)   & \dE b  (z,0)  \\  \dE \bar b  (z,0)  &  i  \dE h (z,0)   \end{pmatrix}, 
  \end{equation}
  and
  \begin{equation}\label{eq:logtn}
    \lim_{n\to\infty} \frac 1 2 \int \ln ( x^2 + t_n^2 ) \dE \nu_{A_n - z} (
    dx) %
    =   \int \ln ( x)   \nu_{\alpha,z} ( dx).
  \end{equation}
  Moreover from theorem \ref{th:girko}, Equation \eqref{eq:unifconvproba},
  lemma \ref{le:girko}, in $\cD' (\dC)$:
  \begin{equation}
    \label{eq:PLmu}
    \Delta  \int \ln ( x)   \nu_{\alpha,z} ( dx) = 2 \pi \mu_\a. 
  \end{equation}
  On the other hand, $ \frac 1 2 \int \ln ( x^2 + t^2 ) \nu_{A_n - z }( dx) = \frac
  {1}{2 n} \ln | \det ( B (z) - it I_{2n} ) ) | $, and from \eqref{eq:Deltat},
  \[
  \Delta  \frac 1 2  \int \ln ( x^2 + t^2 ) \dE \nu_{A_n - z} ,(dx)  = - 2 \partial \dE  b_{1} ( z, i t).
  \]
  The conclusion follows from \eqref{eq:Restn}, \eqref{eq:logtn} and
  \eqref{eq:PLmu}.
\end{proof}

It is possible to compute explicitly the expression \eqref{eq:density} at $z =
0$.

\begin{lem}\label{le:densityat0}
  The density of $\mu_\a$ at $z = 0$ is
  \[
  \frac{1}{\pi}   \frac{\Gamma ( 1 + 1/ \beta)^2 \Gamma ( 1 + \beta)^{ 1/\beta} }{\Gamma ( 1 - \beta) ^{1/\beta}}.
  \]
\end{lem}

\begin{proof} 
  By definition, the real $y_*(0)$ solves the equation
  \[
  1 = \dE \left( \frac{ S }{ y^2 S S' } \right)^\beta = y^{-2\beta} \dE
  S^{-\beta} = \frac{y^{-2\beta} }{\Gamma (\beta)} \int x^{\beta - 1} e^{ -
    \Gamma ( 1- \beta) x^{\beta}} dx .
  \]
  With the change of variable $x \mapsto x^{\beta}$ and the identity $z \Gamma
  (z ) = \Gamma ( 1 + z)$, we find easily, $\dE S^{-\beta} = \left( \Gamma ( 1
    - \beta) \Gamma ( 1 + \beta) \right) ^{-1}$ and
  \[
  y_* (0) =\left(  \Gamma ( 1 - \beta) \Gamma ( 1 + \beta) \right) ^{-\frac{1}{2\beta}}.
  \]
  We also have
  \[
  \dE S^{-1} = \int e^{ - \Gamma ( 1- \beta) x^{\beta}} dx = \frac{1}{\beta
    \Gamma ( 1 - \beta) ^{1/\beta} } \int x^{1/ \beta - 1} e^{-x} dx =
  \frac{\Gamma ( 1 + 1/ \beta) }{\Gamma ( 1 - \beta) ^{1/\beta}},
  \]
  where we have used again the identity $z \Gamma (z ) = \Gamma ( 1 + z)$.
  Then the right hand side of \eqref{eq:density} at $z =0$ is equal to
  \[
  y^2 _* ( 0) y^{-4} _* ( 0) \dE (S S' )^{-1} = y^{-2} _* ( 0)\left( \dE
    S^{-1} \right)^2 .
  \]
\end{proof}

\subsection{Proof of theorem \ref{th:asydensity}}
\label{ss:asydensity}

In this subsection, we prove the last statement of theorem \ref{th:asydensity}
(the first part of the theorem being contained in lemmas \ref{le:density},
\ref{le:densityat0}). We start with a first technical lemma.

\begin{lem}\label{eq:eqYt}
  Let $0 < \beta < 1$, $\delta >0$, and $f$ be a bounded measurable $\dR_+ \to
  \dR$ function such that $f(y) = O ( y^{\beta + \delta} ) $ as $y \downarrow
  0$. Let $Y$ be a random variable such that $ \dP ( Y \geq t ) = L (t) t
  ^{-\beta}$ for some slowly varying function $L$. Then as $t$ goes to
  infinity
  \[
  \dE f \left( \frac Y t \right) %
  \sim \beta L (t) t^{-\beta} \int_0 ^\infty f(y) y^{-\beta -1} dy.
  \]
\end{lem}
\begin{proof}
  Define $Y_t = Y / t$. We fix $\e >0$ and consider the distribution $ \dP (
  Y_t \in \cdot | Y_t \geq \e). $ By assumption, for $s > \e$,
  \[
  \dP  ( Y_t  \geq s  | Y_t \geq \e)  \sim \left( s / \e\right) ^{-\beta}.
  \]
  In particular, the distribution of $Y_t$ given $\{Y_t \geq \e\}$ converges
  weakly as $t$ goes to infinity to the distribution with density $\beta
  x^{-\beta - 1} \e^\beta dx $. Since $f$ is bounded and $L$ slowly varying,
  we get
  \begin{eqnarray*}
    \dE \left[  f \left( \frac Y t \right) \ind_{\{Y \geq \e t\}} \right] & = & \dP (Y_t  \geq \e  )  \dE \left[ f \left( Y_t \right)  \Bigm| Y_t \geq \e\right]  \\
    & \sim &  L(\e t)  \e^{- \beta} t^{-\beta}    \int_\e ^\infty  f(y) \beta y^{-\beta - 1 }  \e^\beta  dy   \\
    & \sim & \beta L (t) t^{-\beta} \int_\e ^\infty f(y) y^{-\beta -1} dy.
  \end{eqnarray*}
  Finally, by assumption, for some constant, $c >0$,
  \begin{eqnarray*}
    \dE \left[  f \left( \frac Y t \right) \ind_{\{Y \leq \e t\}} \right]&  \leq &  c t^{-\beta - \delta}   \dE [ Y^{\beta + \delta}   \ind_{\{Y \leq \e t\}} ] .
  \end{eqnarray*}
  Thus by lemma \ref{le:XV2}, for some new constant $c >0$ and all $t \geq 1 /
  \e$,
  \[
  \dE \left[ f \left( \frac Y t \right) \ind_{\{Y \leq \e t\}} \right] \leq c
  t^{-\beta - \delta} L( \e t) ( \e t) ^{\delta} = c t^{-\beta} L(t)
  \e^{\delta } \frac{L(\e t)}{L(t)} .
  \]
  We may thus conclude by letting $t$ tend to infinity and then $\e$ to
  $0$. 
\end{proof}

\begin{lem}\label{le:zolo}
  Let $S$ be a random variable with Laplace transform \eqref{eq:laplaceZ}.
  There exists a constant $c_0 >0$ such that as $t$ goes to infinity,
  \[
  \dE S^\beta \ind_{\{S \leq t\}} = \ln t + c_0 + o(1). 
  \]
\end{lem}
\begin{proof}
  Let $g_\beta$ be the density function of $S$. From Equation (2.4.8) in
  Zolotarev \cite{zolotarev}, $g_\beta$ has a convergent power series
  representation
  \[
  g_\beta ( x) = \frac 1 \pi \sum_{n=1} ^\infty ( - 1 ) ^{n-1} \frac{ \Gamma (
    n \beta + 1) }{ \Gamma ( n+1) \Gamma ( 1 - \beta )^n } \sin ( \pi n \beta
  ) x^{ - n \beta - 1}.
  \]
  The Stirling formula $\Gamma(x) \sim_{x \to \infty} \sqrt{\frac{2 \pi}{
      x}}~{\left( \frac{x}{e} \right)}^x$ implies that the convergence radius
  of the series is $+\infty$. Recall that $\Gamma(\beta +1) = \beta
  \Gamma(\beta)$, and the Euler reflection formula, $\Gamma( 1- \beta )
  \sin ( \pi \beta ) / \pi = \Gamma( \beta)$. Thus, as $x$ goes to infinity,
  \[
  g_\beta (x) = \beta x^{-\beta -1} + O ( x ^{-2\beta -1}).
  \]
\end{proof}

The next lemma is a consequence of the Karamata Tauberian theorem.
\begin{lem}\label{le:SS'}
  As $t$ goes to infinity, 
  \[
  \dP ( S S' \geq t ) \sim \beta t ^{ -\beta} \ln t ,
  \]
  and, with $c_1 = \beta^2 \int_0^\infty (x+1)^{-2} x^{-\beta} dx $,
  \[
  \dE \frac{SS' }{ ( t + SS' )^2 } \sim c_1  t^{ - 1 - \beta} \ln t.
  \]
\end{lem}

\begin{proof}
  Let $x > 0$, since $S$ and $S'$ are independent we have
  \[
  \dE \exp ( - x SS') = \dE \exp \left( -\Gamma ( 1- \beta)  x^\beta S^{\beta}\right). 
  \]
  From Corollary 8.1.7 in \cite{bingham}, we have as $t$ goes to infinity, $
  \dP ( S > t) \sim t^{-\beta} $. In particular, we have $\dP ( S^\beta > t )
  \sim t^{-1}$ and a new application of Corollary 8.1.7 in \cite{bingham}
  gives as $x \downarrow 0$,
  \[
  1 - \dE \exp ( - x S^\beta ) \sim x \ln x^{-1}. 
  \]
  We obtain 
  \[
  1 - \dE \exp ( - x SS') \sim \Gamma ( 1- \beta) x^\beta \ln ( \Gamma ( 1-
  \beta) x^{-\beta} ) \sim \beta \Gamma ( 1- \beta) x^\beta \ln ( x^{-1} ).
  \]
  We then conclude by a third application of Corollary 8.1.7 in
  \cite{bingham}. The second statement is a consequence of lemma
  \ref{eq:eqYt}. 
\end{proof}

The next lemma gives the asymptotic behavior of $y_*(r)$ as $r$ goes to
infinity.
\begin{lem}\label{le:laplace}
  There exists a constant $c_2 >0$ such that as $r$ goes to infinity,
  \[
  y_* ( r) \sim c_2 \sqrt r e^{- r ^\beta /2 }.
  \]
\end{lem}

\begin{proof}
  From Equations \eqref{eq:gammaLaplace}, \eqref{eq:fixpty}, we have with $y_*
  = y_* (r)$,
\begin{eqnarray}
  1  & = & \frac{1}{\Gamma ( \beta) } \int x^{\beta-1} \dE \exp \left( - \frac{x r }{S} - x y_* ^2 S' \right) dx \nonumber  \\
  & = & \frac{1}{\Gamma ( \beta) } \int x^{\beta-1}  e^{- x^\beta y_* ^{2\beta} \Gamma ( 1 - \beta) } \dE e^{ - \frac{x r}{S} }   dx  \nonumber \\
  & =&  \frac{1}{\Gamma ( 1 + \beta) \Gamma ( 1 - \beta) y_*^{2\beta} } \int   e^{- x } \dE e^{ - \frac{x^{1/\beta} r y_*^{-2} }{S  \Gamma ( 1 - \beta)^{1/\beta}  } }  dx. \label{eq:laplacey}
\end{eqnarray}
By lemma \ref{le:yC1}, $\lim_{r \to \infty} y_*(r) = 0$. Hence, from the above
expression, we deduce that the term $r y_*^{-2}$ goes to infinity as $r$ goes
to infinity. Define
\begin{eqnarray*}
I(y) & = & \frac{1}{\Gamma ( 1 + \beta) \Gamma ( 1 - \beta)  } \int   e^{- x }  e^{ - \frac{x^{1/\beta} }{y  \Gamma ( 1 - \beta)^{1/\beta}  } }  dx = I_0 ( y ) + I_1 ( y) + I_2 ( y), 
\end{eqnarray*}
with $I_0 (y) = I(y) \ind_{\{y \geq 1\}}$, 
\begin{eqnarray*}
  I_1 (y) & = & \frac{\ind_{\{y \leq 1\}}}{\Gamma ( 1 + \beta) \Gamma ( 1 - \beta)  } \int e^{ - \frac{x^{1/\beta} }{ y  \Gamma ( 1 - \beta)^{1/\beta}  } }  dx  = y^\beta \ind_{\{y \leq 1\}},\\
  I_2(y)  & = &  \frac{\ind_{\{y \leq 1\}}}{\Gamma ( 1 + \beta) \Gamma ( 1 - \beta)  } \int   (e^{- x } - 1) e^{ - \frac{x^{1/\beta} }{ y  \Gamma ( 1 - \beta)^{1/\beta}  } }  dx  .
\end{eqnarray*}
The function $I$ is increasing and $\lim_{y \to \infty} I(y) < \infty$. Also, the function $I_0$ is equal to  $0$ in a neighborhood of $0$. By lemma \ref{eq:eqYt}, we get as $t$ goes to infinity, 
\[
\dE I_0 ( S / t ) \sim a_0 t^{-\beta},
\]
for some positive constant $a_0 = \frac{1}{\Gamma ( 1 + \beta) \Gamma ( 1 - \beta)  }  \int_1 ^\infty \int   e^{- x } e^{ - \frac{x^{1/\beta} }{y  \Gamma ( 1 - \beta)^{1/\beta}  } }  \beta y^{-\beta -1} dx dy $. By lemma \ref{le:zolo}, 
\[
\dE [ I_1 ( S / t ) ]  =  t^{-\beta} \ln  t + c_0 t^{-\beta}  + o(1).  
\]
Also, from Laplace method, $I_2 ( y) \sim -  \Gamma(  2 \beta  ) \Gamma ( 1 - \beta) ^2  y^{2\beta}$ as $y$ goes to $0$. By lemma \ref{eq:eqYt},
\[
\dE I_2 ( S / t ) \sim a_2 t^{-\beta},
\]
with $a_2 = \frac{1}{\Gamma ( 1 + \beta) \Gamma ( 1 - \beta) } \int_0 ^1 \int
(e^{- x } -1) e^{ - \frac{x^{1/\beta} }{y \Gamma ( 1 - \beta)^{1/\beta} } }
\beta y^{-\beta -1} dx dy $. Hence, for $t = r y_*^{-2}$, we get from
\eqref{eq:laplacey}
\[
y_*^{2\beta}  =   (r y_*^{-2} ) ^{-\beta}  \ln  ( r y^{-2}_* )  + (c_0 + a_0 + a_2) (r y_*^{-2} ) ^{-\beta}  + o ( (r y_*^{-2} ) ^{-\beta} ) .
\]
In other words, 
\[
r^{\beta}  =    \ln  ( r y^{-2}_* )   +  (c_0 + a_0 + a_2)  + o ( 1).
\]
We conclude by setting $c_2 = \exp ( (c_0 + a_0 + a_2) /2)$. 
\end{proof}

\begin{lem}\label{le:laplace2}
As $r$ goes to infinity, 
\[
y' (r) \sim - c^{-1}_3 y_* ( r) r^{ \beta - 1}, 
\]
where $c_3 = 2  \int_0^\infty \int_0^\infty x   e^{- x}  e^{ - \frac{x^{1/\beta}  }{s \Gamma ( 1 - \beta) ^{1/ \beta} } }   \beta s^{-\beta -1} dx ds  / \left(   \Gamma (1+ \beta) \Gamma ( 1 - \beta) \right) $. 
\end{lem}

\begin{proof}
  We define
  \[
  G (y , r) = \dE \left(  \frac{    S    }{ r  +   y^2   S S'    }   \right)^\beta = \frac{1}{\Gamma ( \beta) }  \int x^{\beta-1}  e^{- x^\beta y ^{2\beta} \Gamma ( 1 - \beta) } \dE e^{ - \frac{x r}{S} }   dx.
  \]
  From the implicit function theorem 
  \[
  y'_* ( r) = - \frac{\partial_ r G (y_* , r) } {\partial_y  G (y_* , r) }. 
  \]
  We have 
  \begin{eqnarray*}
    \partial_ y G (y,r) & = &   - \frac{2 \beta   \Gamma ( 1 - \beta) y^{ 2 \beta -1} }{\Gamma ( \beta) }  \int x^{2\beta-1}  e^{- x^\beta y ^{2\beta} \Gamma ( 1 - \beta) } \dE e^{ - \frac{x r}{S} }   dx \\
    & = & -  \frac{2}{ y^{2\beta +1} \Gamma (1+ \beta) \Gamma ( 1 - \beta) }  \int x   e^{- x}  \dE e^{ - \frac{x^{1/\beta}  r y ^{-2} }{S \Gamma ( 1 - \beta) ^{1/ \beta} } }   dx \\
  \end{eqnarray*}
  The Laplace method implies that, as $t$ goes infinity,
  \[
  \int x e^{- x } e^{ - \frac{x^{1/\beta} t }{\Gamma ( 1 - \beta) ^{1/ \beta}
    } } dx \sim \Gamma( 2 \beta ) \Gamma ( 1 - \beta) ^2 t^{-2\beta}.
  \]
  Thus by lemma \ref{eq:eqYt}, we deduce that 
  \[
  \int x   e^{- x}  \dE  e^{ - \frac{x^{1/\beta}  t }{S \Gamma ( 1 - \beta) ^{1/ \beta} } }   dx \sim  t^{-\beta} \int \int x   e^{- x}  e^{ - \frac{x^{1/\beta}  }{s \Gamma ( 1 - \beta) ^{1/ \beta} } }   \beta s^{-\beta -1} dx ds \sim c t^{-\beta}. 
  \]
  Applying the above to $t = r y_*^{-2} (r)$ we deduce, with $c_3 = 2 c / (
  \Gamma (1+ \beta) \Gamma ( 1 - \beta) ) $,
  \[
  \partial_ y G (y_*,r) \sim - c_3  r^{-\beta } y_*^{-1} (r). 
  \]
  Similarly, the derivative of $G$ with respect to $r$ is  
  \begin{eqnarray*}
    \partial_ r G (y,r) 
    & = &   - \frac{1}{ y^{2\beta +2} \Gamma (1- \beta)^{1/\beta +1} \Gamma ( 1 + \beta) }  \int x^{1 / \beta}  e^{- x}  \dE e^{ - \frac{x^{1/\beta}  r y ^{-2} }{S \Gamma ( 1 - \beta) ^{1/ \beta} } }  S^{-1}  dx .\\
  \end{eqnarray*}
  Once again, Laplace method implies that, as $t$ goes infinity,
  \[
  \int x^{1 / \beta}  e^{- x}  e^{ - \frac{x^{1/\beta}  t }{\Gamma ( 1 - \beta) ^{1/ \beta} } }   dx   \sim  \Gamma(  \beta +1 ) \Gamma ( 1 - \beta) ^{1/ \beta+1}  t^{-\beta-1}.
  \]
  In particular, for all $\e >0$ there exists $t_0$ such that 
  \begin{align*}
    & ( 1- \e)    t^{-\beta-1} \dE S^\beta \ind_{\{S \leq t / t_0\}} \\
    & \leq  \frac{1}{  \Gamma (1- \beta)^{1/\beta +1} \Gamma ( 1 + \beta) } \int x^{1 / \beta}  e^{- x}  \dE e^{ - \frac{x^{1/\beta}  t  }{S \Gamma ( 1 - \beta) ^{1/ \beta} } }  S^{-1} \ind_{\{S \leq t / t_0\}}  dx \\
    &  \leq ( 1+ \e)    t^{-\beta-1} \dE S^\beta \ind_{\{S \leq t / t_0\}} .
  \end{align*}
  By lemma \ref{le:zolo},
  \[
  \dE S^\beta \ind_{\{S \leq t / t_0\}} \sim \ln t. 
  \]
  It follows that for some $t_1 > t_0$ and all $t \geq t_1$, 
  \[
  ( 1- 2\e) t^{-\beta-1} \ln t \leq \frac{1}{ \Gamma (1- \beta)^{1/\beta +1}
    \Gamma ( 1 + \beta) } \int x^{1 / \beta} e^{- x} \dE e^{ - \frac{x^{1/\beta}
      t }{S \Gamma ( 1 - \beta) ^{1/ \beta} } } S^{-1} \ind_{\{S \leq t / t_0\}}
  dx \leq ( 1+ 2\e) t^{-\beta-1} \ln t .
  \]
  On the other hand, for some constant $c >0$ and all $t \geq 1$, 
  \[
  \int x^{1 / \beta} e^{- x} \dE e^{ - \frac{x^{1/\beta} t }{S \Gamma ( 1 -
      \beta) ^{1/ \beta} } } S^{-1} \ind_{\{S \geq t / t_0\}} dx \leq \int x^{1
    / \beta} e^{- x} dx \dP ( S \geq t /t_0 ) \leq c t^{-\beta-1} t_0^{\beta+1}.
  \]
  We thus have proved that 
  \[
  \frac{1}{  \Gamma (1- \beta)^{1/\beta +1} \Gamma ( 1 + \beta) } \int x^{1 / \beta}  e^{- x}  \dE e^{ - \frac{x^{1/\beta}  t  }{S \Gamma ( 1 - \beta) ^{1/ \beta} } }  S^{-1}   dx   \sim t^{-\beta-1} \ln t,
  \]
  and
  \[
  \partial_ r G (y_*(r),r)  \sim -  r^{-\beta - 1} \ln ( r y^{-2} ) \sim -  r^{-1}.
  \]
  The statement follows.
\end{proof}

\begin{proof}[Proof of theorem \ref{th:asydensity}]
  From Equation \eqref{eq:density} and lemma \ref{le:SS'}, the density at $r =
  |z|^2$ is equivalent to $1 /  \pi$ times
  \[ 
  \left(1 - 2  r  \frac{ y' _* (r)}{ y _* ( r) }  \right)  y^{-2}_ * ( r) c_1 (
  r y_* ^{-2} ) ^{-1 -\beta} \ln ( r y_*^{-2} ) . 
  \]
  It remains to apply lemmas \ref{le:laplace} and \ref{le:laplace2}, and set
  the multiplicative constant to be $c = 2 \pi^{-1} c_3^{-1} c_1 c_2^{2 \beta}
  $.
\end{proof}

\appendix

\section{Logarithmic potentials and Hermitization}\label{se:loppot}

Let $\mathcal{P}(\dC)$ be the set of probability measures on $\dC$ which
integrate $\ln\ABS{\cdot}$ in a neighborhood of infinity. For every
$\mu\in\mathcal{P}(\dC)$, the \emph{logarithmic potential} $U_\mu$ of $\mu$ on
$\dC$ is the function $U_\mu:\dC\to[-\infty,+\infty)$ defined for every
$z\in\dC$ by
\begin{equation}\label{eq:logpot}
  U_\mu(z)=  \int_{\dC}\!\ln|z-z'|\,\mu(dz')
  =(\ln\ABS{\cdot}*\mu)(z).
\end{equation}
Note that in classical potential theory, the definition is opposite in sign,
but ours turns out to be more convenient (lightweight) for our purposes. Since
$\ln\ABS{\cdot}$ is Lebesgue locally integrable on $\dC$, one can check by
using the Fubini theorem that $U_\mu$ is Lebesgue locally integrable on $\dC$.
In particular, $U_\mu<\infty$ a.e. (Lebesgue almost everywhere) \ and
$U_\mu\in\mathcal{D}'(\dC)$. Since $\ln\ABS{\cdot}$ is the fundamental
solution of the Laplace equation in $\dC$, we have, in $\mathcal{D}'(\dC)$,
\begin{equation}\label{eq:lap}
  \Delta U_\mu= 2\pi \mu.
\end{equation}

\begin{lem}[Unicity]\label{le:unicity}
  For every $\mu,\nu\in\mathcal{P}(\dC)$, if $U_\mu=U_\nu$ a.e.\ then
  $\mu=\nu$.
\end{lem}

\begin{proof}
  Since $U_\mu=U_\nu$ in $\mathcal{D}'(\dC)$, we get $\Delta
  U_\mu=\Delta U_\nu$ in $\mathcal{D}'(\dC)$. Now \eqref{eq:lap} gives
  $\mu=\nu$ in $\mathcal{D}'(\dC)$, and thus $\mu=\nu$ as measures
  since $\mu$ and $\nu$ are Radon measures.
\end{proof}

If $A$ is an $n\times n$ complex matrix and $P_A(z):=\det(A-zI)$ is its
characteristic polynomial,
\[
U_{\mu_A}(z)
=\int_{\dC}\!\ln\ABS{z'-z}\,\mu_A(dz')
=\frac{1}{n}\ln\ABS{\det(A-zI)}
=\frac{1}{n}\ln\ABS{P_A(z)}
\]
for every $z\in\dC\setminus\{\lambda_1(A),\ldots,\lambda_n(A)\}$. We
have also the alternative expression 
\begin{equation}\label{eq:LPVS}
U_{\mu_A}(z)
=\frac{1}{n}\ln\det(\sqrt{(A-zI)(A-zI)^*})
=\int_0^\infty\!\ln(t)\,\nu_{A-zI}(dt).
\end{equation}
The identity above bridges the eigenvalues with the singular values, and is at
the heart of the following lemma, which allows to deduce the convergence of
$\mu_A$ from the one of $\nu_{A-zI}$. The strength of this Hermitization lies
in the fact that contrary to the eigenvalues, one can control the singular
values with the entries of the matrix. The price paid here is the introduction
of the auxiliary variable $z$ and the uniform integrability. We recall that on
a Borel measurable space $(E,\mathcal{E})$, we say that a Borel function
$f:E\to\mathds{R}$ is uniformly integrable for a sequence of probability
measures $(\eta_n)_{n\geq1}$ on $E$ when
\[
\lim_{t\to\infty}\varlimsup_{n\to\infty}\int_{\{|f|>t\}}\!|f|\,d\eta_n=0.
\]
We will use this property as follows: if $\eta_n\weak\eta$ and $f$ is
continuous and uniformly integrable for $(\eta_n)_{n\geq1}$ then $f$ is
$\eta$-integrable and $\lim_{n\to\infty}\int\!f\,d\eta_n=\int\!f\,\eta$.
Similarly for a sequence random probability measures $(\eta_n)_{n \geq 1}$ we
will say that $f$ is uniformly integrable for $(\eta_n)_{n\geq1}$ in
probability, if for all $\e >0$
\[
\lim_{t\to\infty} \varlimsup_{n\to\infty} 
\dP \left( \int_{\{|f|>t\}}\!|f|\,d\eta_n > \e \right) =0.
\]
A proof of lemma \ref{le:girko} below can be found in \cite{cirmar} which covers the ``a.s.'' case, the ``in probability'' case being similar. It relies
only on the unicity lemma \ref{le:unicity}, the classical Prohorov theorem,
and the Weyl inequalities of Lemma \ref{le:weyl} linking eigenvalues and
singular values. 

\begin{lem}[Girko's Hermitization method]\label{le:girko}
  Let $(A_n)_{n\geq1}$ be a sequence of complex random matrices where $A_n$ is
  $n\times n$ for every $n\geq1$. Suppose that for Lebesgue almost all $z\in\dC$, there exists a probability measure $\nu_z$ on $[0,\infty)$ such
  that 
  \begin{itemize}
  \item[(i)] a.s.\ $\PAR{\nu_{A_n-zI}}_{n\geq1}$ tends weakly to $\nu_{z}$ 
  \item[(ii)] a.s.\ (resp. in probability) $\ln(\cdot)$ is uniformly integrable for
    $\PAR{\nu_{A_n-zI}}_{n\geq1}$
  \end{itemize}
  Then there exists a probability measure $\mu\in\mathcal{P}(\dC)$ such
  that 
  \begin{itemize}
  \item[(j)]  a.s.\ (resp. in probability) $\PAR{\mu_{A_n}}_{n\geq1}$ converges weakly to $\mu$ 
  \item[(jj)] for a.a.\ $z\in\dC$,
    \[
    U_\mu(z)=\int_0^\infty\!\ln(t)\,\nu_z(dt).
    \]
  \end{itemize}
\end{lem}

\section{General spectral estimates}\label{se:dema}

\begin{lem}[Basic inequalities \cite{MR1288752}]\label{le:basic}
If $A$ and $B$ are $n\times n$ complex matrices then
\begin{equation}\label{eq:basic0}
  s_1(AB)\leq s_1(A)s_1(B)
  \quad\text{and}\quad
  s_1(A+B)\leq s_1(A)+s_1(B)
\end{equation}
and
\begin{equation}\label{eq:basic1}
  \max_{1\leq i\leq n}|s_i(A)-s_i(B)|\leq s_1(A-B).
\end{equation}
\end{lem}

\begin{lem}[Rudelson-Vershynin row bound \cite{MR2407948,cirmar}]
  \label{le:rvdist}
  Let $A$ be a complex $n\times n$ matrix with rows $R_1,\ldots,R_n$. Define
  the vector space $R_{-i}:=\mathrm{span}\{R_j;j\neq i\}$. We have then
  \[
  n^{-1/2}\min_{1\leq i\leq n}\mathrm{dist}(R_i,R_{-i}) \leq
  s_n(A) \leq \min_{1\leq i\leq n}\mathrm{dist}(R_i,R_{-i}).
  \]
\end{lem}

Recall that the singular values $s_1(A),\ldots,s_{n'}(A)$ of a rectangular
$n'\times n$ complex matrix $A$ with $n'\leq n$ are defined by
$s_i(A):=\lambda_i(\sqrt{AA^*})$ for every $1\leq i\leq n'$.

\begin{lem}[Tao-Vu negative second moment {\cite[Lemma A4]{MR2722794}}]\label{le:tvneg}
  If $A$ is a full rank $n'\times n$ complex matrix ($n'\leq n$) with rows
  $R_1,\ldots,R_{n'}$, and $R_{-i}:=\mathrm{span}\{R_j;j\neq i\}$, then
   \[ 
   \sum_{i=1}^{n'}s_i(A)^{-2}=\sum_{i=1}^{n'}\mathrm{dist}(R_i,R_{-i})^{-2}.
   \]  
\end{lem}

\begin{lem}[Cauchy interlacing by rows deletion \cite{MR1288752}]
  \label{le:cauchy}
  Let $A$ be an $n\times n$ complex matrix. If $B$ is $n'\times n$, obtained
  from $A$ by deleting $n-n'$ rows, then for every $1\leq i\leq n'$,
  \[
  s_i(A)\geq s_i(B)\geq s_{i+n-n'}(A).
  \]
\end{lem}

\begin{lem}[Weyl inequalities \cite{MR0030693}]\label{le:weyl}
  For every $n\times n$ complex matrix $A$, we have 
  \begin{equation}\label{eq:weyl0}
    \prod_{i=1}^k|\lambda_i(A)|\leq \prod_{i=1}^ks_i(A)
    \quad\text{and}\quad
    \prod_{i=k}^ns_i(A) \leq \prod_{i=k}^n|\lambda_i(A)| 
  \end{equation}
  for all $1\leq k\leq n$. In particular, by viewing $\ABS{\det(A)}$ as a
  volume,
  \begin{equation}\label{eq:weyl1}
    |\det(A)|=\prod_{k=1}^n|\lambda_k(A)|=\prod_{k=1}^ns_k(A)
    =\prod_{k=1}^n\mathrm{dist}(R_k,\mathrm{span}\{R_1,\ldots,R_{k-1}\})
  \end{equation}
  where $R_1,\ldots,R_n$ are the rows of $A$. Moreover, for every increasing
  function $\varphi$ from $(0,\infty)$ to $(0,\infty)$ such that
  $t\mapsto\varphi(e^t)$ is convex on $(0,\infty)$ and
  $\varphi(0):=\lim_{t\to0^+}\varphi(t)=0$, we have
  \begin{equation}\label{eq:weyl2}
    \sum_{i=1}^k\varphi(|\lambda_i(A)|^2) \leq \sum_{i=1}^k\varphi(s_i(A)^2)
  \end{equation}
  for every $1\leq k\leq n$. In particular, with $\varphi(t)=t^{r/2}$, $r>0$,
  and $k=n$, we obtain
  \begin{equation}\label{eq:weyl3}
    \sum_{k=1}^n|\lambda_k(A)|^r \leq \sum_{k=1}^ns_k(A)^r.
  \end{equation}
\end{lem}

\begin{lem}[Schatten bound {\cite[proof of Theorem 3.32]{zhan}}]
  Let $A$ be an $n\times n$ complex matrix with rows
  $R_1,\ldots,R_n$. Then for every $0<r\leq 2$,
  \begin{equation}\label{eq:zhan}
    \sum_{k=1}^n s_k(A)^r 
    \leq 
    \sum_{k=1}^n \NRM{R_k}_2^r.
  \end{equation}
\end{lem}

\section{Additional lemmas}

We begin with a lemma on truncated moments. We skip the proof since it follows 
from an adaptation of the proof in the real case given by e.g.\ Feller \cite[Theorem VIII.9.2]{Feller}.

\begin{lem}[Truncated moments]\label{le:XV2}
  If (H1) holds then for every $p > \alpha $,
  \[
  \dE\SBRA{| X_{11} |^p\ind_{\{| X_{11} | \leq t\}}} %
  \sim c (p) L(t) t^{p-\alpha}
  \]
  where $c(p) := \alpha/(p-\alpha)$. In particular, we have
  \[ 
  \dE\SBRA{| X_{11} | ^p \ind_{\{| X_{11} | \leq a_n\}}} %
  \sim c(p) \frac{a_n^{p}}{n}.
  \]
\end{lem}

We end up this section by a result on the concentration of the spectral
measure of Hermitian or Hermitized random matrices, mentioned in \cite{BLS}.
The total variation norm of $f:\dR \to\dR $ is
\[
\NRM{f}_\textsc{TV}:=\sup \sum_{k \in \dZ} | f(x_{k+1})-f(x_k) |, \, 
\]
where the supremum runs over all sequences $(x_k)_{k \in \dZ}$ such that
$x_{k+1} \geq x_k$ for any $k \in \dZ$. If $f = \mathds{1}_{(-\infty,s]}$ for
some real $s$ then $\NRM{f}_\textsc{TV}=1$, while if $f$ has a derivative in
$\mathrm{L}^1(\dR)$, we get
\[
\NRM{f}_\textsc{TV}=\int_{\dR}\!|f'(t)|\,dt.
\] 

The following lemma comes with remarkably weak assumptions, and allows to
deduce the almost sure weak convergence of empirical spectral measures of
random matrices without any moment assumptions on the entries. We discovered
that this lemma was obtained independently by Guntuboyina and Leeb in
\cite{MR2535081}, where they discuss the relationships with more classical
results.

\begin{lem}[Concentration for spectral measures]\label{le:concspec}
  Let $H$ be an $n\times n$ random Hermitian matrix. Let us assume that the
  vectors $(H_i)_{1 \leq i \leq n}$, where $H_i := (H_{ij})_{1 \leq j \leq i}
  \in \dC^i$, are independent. Then for any $f:\dR\to\dR$ 
  going to $0$ at $\pm\infty$ and
  such that
  $\NRM{f}_\textsc{TV}\leq1$ 
  and every $t\geq0$,
  \[
  \dP \PAR{\ABS{\int\!f\,d\mu_H -\dE\int\!f\,d\mu_H} \geq t} %
  \leq 2 \exp\PAR{-\frac{n t^2}{2}}.
  \]
  Similarly, if $M$ is an $n\times n$ complex random matrix with independent
  rows (or with independent columns) then for any 
  $f:\dR_+\to\dR$ going to $0$ at $+\infty$ 
  with $\NRM{f}_\textsc{TV}\leq1$ and every $t\geq0$,
  \[
  \dP\PAR{\ABS{\int\!f\,d\nu_M -\dE\int\!f\,d\nu_M} \geq t} %
  \leq 2 \exp\PAR{- 2n t^2}.
  \]
\end{lem}

\begin{proof}
  We prove only the Hermitian version, the non-Hermitian version being
  entirely similar. Let us start by showing that for every $n\times n$
  deterministic Hermitian matrices $A$ and $B$ and any measurable function $f$
  with $\NRM{f}_\textsc{TV} = 1$,
  \begin{equation}\label{eq:rankineq}
    \ABS{\int\!f\,d\mu_A-\int\!f\,d\mu_{B}} \leq \frac{\mathrm{rank}(A-B)}{n}.
  \end{equation}
  Indeed, it is well known (follows from interlacing, see e.g.\
  \cite{MR0407051} or \cite[Theorem 11.42]{bai-silverstein-book}) that
  \[
  \NRM{F_A- F_B}_\infty\leq\frac{\mathrm{rank}(A-B)}{n}
  \]
  where $F_A$ and $F_B$ are the cumulative distribution functions of $\mu_A$
  and $\mu_B$ respectively. Now if $f$ is smooth, we get, by integrating by
  parts,
  \[
  \ABS{\int\!f\,d\mu_A-\int\!f\,d\mu_B}
  =\ABS{\int_\dR \!f'(t)F_A(t)\,dt-\int_\dR \!f'(t)F_B(t)\,dt} %
  \leq \frac{\mathrm{rank}(A-B)}{n}\int_\dR\!|f'(t) |\,dt,
  \]
  and since the left hand side depends on at most $2n$ points, we get
  \eqref{eq:rankineq} by approximating $f$ by smooth functions. Next, for any
  $x = (x_1, \ldots, x_n) \in\cX:= \{(x_i)_{1 \leq i \leq n} : x_i \in
  \dC^{i-1} \times \dR\}$, let $H(x)$ be the $n \times n$ Hermitian
  matrix given by $H(x)_{ij} := x_{i,j}$ for $1 \leq j \leq i \leq n$. We have
  $\mu_H = \mu_{H(H_1, \ldots, H_n)}$. For all $x \in \cX$ and $x'_i \in
  \dC^{i-1} \times \dR $, the matrix
  \[ 
  H(x_1,\ldots,x_{i-1}, x_i , x_{i+1}, \ldots , x_n) %
  - H(x_1,\ldots,x_{i-1}, x'_i , x_{i+1}, \ldots , x_n) 
  \]
  has only the $i$-th row and column possibly different from $0$, and thus
  \[
  \mathrm{rank}\PAR{H(x_1,\ldots,x_{i-1}, x_i , x_{i+1}, \ldots , x_n) %
    - H(x_1,\ldots,x_{i-1}, x'_i , x_{i+1}, \ldots , x_n)}\leq 2.
  \]
  Therefore from \ref{eq:rankineq}, we obtain, for every
  $f:\dR \to\dR $ with $\NRM{f}_\textsc{TV}\leq 1$,
  \[
  \ABS{\int\!f\,d\mu_{H(x_1,\ldots,x_{i-1}, x_i , x_{i+1}, \ldots , x_n)} %
    -\int\!f\,d\mu_{H(x_1,\ldots,x_{i-1}, x'_i , x_{i+1}, \ldots , x_n)}} %
  \leq \frac{2}{n}.
  \]
  The desired result follows now from the Azuma--Hoeffding inequality, see
  e.g.\ \cite[Lemma 1.2]{mcdiarmid}. 
\end{proof}

\providecommand{\bysame}{\leavevmode\hbox to3em{\hrulefill}\thinspace}


\end{document}